\documentclass[a4paper,reqno]{amsart}
\usepackage{amsmath, amscd, a4, amssymb, latexsym, amsthm, xspace, color}
\usepackage{ifpdf}
\usepackage{multirow}
\usepackage{tikz}
\usetikzlibrary{arrows}
\usepackage{breqn}
\usepackage{paralist}
\usepackage{graphicx}
\usepackage[all]{xy}
\usepackage{caption} 
\usepackage[draft]{todonotes}

\newcommand{\Teichmuller}{Teich\-m\"uller\xspace}

\newcommand{\ZZ}{\mathbb{Z}}

\newcommand{\CC}{\mathbb{C}}

\newcommand{\PP}{\mathbb{P}}
\newcommand{\NN}{\mathbb{N}}

\newcommand{\HH}{\mathbb{H}}

\newcommand{\QQ}{\mathbb{Q}}
\newcommand{\RR}{\mathbb{R}}

\renewcommand{\r}{r}

\newcommand{\Jac}{{\rm Jac} \,}

\newcommand{\Gal}{{\rm Gal}}

\newcommand{\fraco}{\mathfrak{o}}

\newcommand{\cEE}{{\mathcal E}}

\newcommand{\cNN}{{\mathcal N}}

\newcommand{\SL}{{\rm SL}}

\newcommand{\GL}{{\rm GL}}

\newtheorem{Defi}{Definition}[section]

\newtheorem{Rem}[Defi]{Remark}

\newtheorem{Prop}[Defi]{Proposition}
\newtheorem{Lemma}[Defi]{Lemma}
\newtheorem{Cor}[Defi]{Corollary}
\newtheorem{Thm}[Defi]{Theorem}

\newtheorem{Conv}{Convention}

\newcommand{\odd}{{\rm odd}}

\newcommand{\hyp}{{\rm hyp}}

\newcommand{\red}{{\rm red}}
\DeclareMathOperator{\Tr}{Tr}
\newcommand{\moduli}[1][g]{{\mathcal M}_{#1}}
\newcommand{\barmoduli}[1][g]{{\overline{\mathcal M}}_{#1}}

\newcommand{\omoduli}[1][g]{{\Omega\mathcal M}_{#1}}

\newcommand{\Prym}[1][g]{{\rm Prym}_{#1}}

\def\={\;=\;}  \def\+{\,+\,} \def\m{\,-\,}

\def\be{\begin{equation}}   \def\ee{\end{equation}}     \def\bes{\begin{equation*}}    \def\ees{\end{equation*}}
\def\ba{\be\begin{aligned}} \def\ea{\end{aligned}\ee}   \def\bas{\bes\begin{aligned}}  \def\eas{\end{aligned}\ees}

\usepackage[style=alphabetic,
            isbn=false,
            doi=false,
            url=false,
            backend=bibtex, maxnames = 5
           ]{biblatex}
\defbibheading{bibliography}[\bibname]{%
  \section*{References}%
}
\DeclareFieldFormat[article]{title}{{\it #1}} 
\DeclareFieldFormat{journaltitle}{{\rm #1}} 
\renewbibmacro{in:}{%
  \ifentrytype{article}{}{\printtext{\bibstring{in}\intitlepunct}}}

\bibliography{my}

\begin{document}{\large}
\title[Teichm\"uller  curves in Prym loci]
{Non-existence and finiteness results for \\ Teichm\"uller curves 
in Prym loci}

\date{\today}
\author{Erwan Lanneau and Martin M\"oller}

\address{
UMR CNRS 5582 \newline
Univ. Grenoble Alpes, CNRS, Institut Fourier, F-38000 Grenoble, France}
\email{erwan.lanneau@univ-grenoble-alpes.fr}

\address{Institut fur Mathematik, Goethe-Universit\"at Frankfurt, 
Robert-Mayer-Str. 6-8 60325 Frankfurt am Main, Germany
}
\email{moeller@math.uni-frankfurt.de}

\begin{abstract}
The minimal stratum in Prym loci have been the first source of
infinitely many primitive, but not algebraically primitive
\Teichmuller curves. We show that the stratum Prym(2,1,1) contains
no such \Teichmuller curve and the stratum Prym(2,2) at most $92$
such \Teichmuller curves. 
\par
This complements the recent progress establishing  general
-- but non-effective --
methods to prove finiteness results for \Teichmuller curves and
serves as proof of concept how to use the torsion condition
in the non-algebraically primitive case.
\end{abstract}

\maketitle
\setcounter{tocdepth}{1}
\tableofcontents


\section{Introduction}

\Teichmuller curves are isometrically immersed algebraic curves
$C = \HH/\Gamma \to \moduli[g]$ in the moduli space of genus $g$ Riemann
surfaces. These arise as $\SL_2(\RR)$-orbits of special flat 
surfaces $(X,\omega)$ or half-translation surfaces $(X,q)$ that are 
called Veech surfaces.
By a canonical double covering construction half-translation surfaces can be
reduced to flat surfaces, the classification problem for \Teichmuller curves
is primarily focused on those generated by flat surfaces. We will exclusively 
deal with this case in the sequel.
\par
A Veech surface is called {\em primitive}, if it is does not arise 
from a flat Veech surface of lower genus via a covering construction.
The {\em trace field} $K = \QQ(\Tr(\gamma) : \gamma \in \Gamma)$ is a 
useful invariant of a Veech surface. If $[K:\QQ] = g$, then the Veech 
surface is called {\em algebraically primitive}. This notion implies that 
the Veech surface is primitive, but the converse does not hold, examples 
being given by \Teichmuller curves in Prym loci that we define below.
\Teichmuller curves are called (algebraically) primitive, if the
generating Veech surfaces have this property.
\par
The classification problem for \Teichmuller curves can be subdivided
into the classification of primitive \Teichmuller curves and the
classification of covering constructions.
Current progress towards the classification of \Teichmuller curves
consists of results in three different flavors. 
\par
First, for low genus and Veech surfaces with a single 
zero, there are complete classification
results, using the geometry of prototypes. This applies to the 
classification of primitive \Teichmuller curves in genus two 
(\cite{mcmullenspin}) and in the Prym loci (\cite{LNWeierstrass}), 
see Section~\ref{sec:Prymloci} for their definition.
\par
Second, there are (in principle effective) finiteness results for algebraically
primitive \Teichmuller curves in the hyperelliptic components of
$\omoduli[g](g-1,g-1)$ (\cite{Mo08}) and in genus three (\cite{BHM14}). 
However, only for $\omoduli[2](1,1)$ there is a complete classification 
(\cite{mcmullentor}). In the other cases, the theoretical bounds
given by the proofs cannot be directly translated into feasible algorithms
and should be combined with techniques recently developed.
\par
The third group of results consists of non-effective finiteness 
theorems based on equidistribution results of Eskin, Mirzakhani
and Mohammadi (\cite{esmi} and \cite{esmimo}). These methods apply e.g.\
to the algebraically primitive case in prime genus with a single zero 
(\cite{MatWri13}).  The most general result in this direction was
proven recently by Eskin, Filip and Wright (\cite{EFW17}) who show
in every genus the finiteness of the number of \Teichmuller curves 
with trace field of degree greater than two and more generally 
finiteness outside
special affine invariant submanifolds. This is complemented by
the finiteness of \Teichmuller curves in non-arithmetic rank-one
orbit closures proven in \cite{LNW}.
Yet another route towards non-effective finiteness results is taken
by Hamenst\"adt (\cite{hamtypical}).
\par
\medskip
In this paper we examine two loci in genus three
that could, in the light of the results above, potentially contain
infinitely many primitive (but not algebraically primitive) \Teichmuller 
curves. The loci under consideration belong to strata with several
zeros, so that the torsion condition from \cite{moellerper} gives additional
constraints on the existence of such \Teichmuller curves.
The main purpose here is to show how to use this torsion condition
in the non-algebraically primitive case to prove effective finiteness
results.
\par
\begin{Thm}
\label{thm:main}
There is at most $92$ primitive \Teichmuller curves lying inside 
the Prym locus $\Prym[3](2,2)$ in $\omoduli[3]$.
The possible examples have trace fields $\QQ[\sqrt{2}]$, $\QQ[\sqrt{3}]$ or $\QQ[\sqrt{33}]$.
\end{Thm}
\par
In a stratum with one more zero, we push the argument to a
complete classification.
\par
\begin{Thm} \label{thm:main:2}
The Prym locus $\Prym[3](2,1,1)$ in $\omoduli[3]$ does not contain any
primitive \Teichmuller curve.
\end{Thm}
\par
We  briefly  sketch  the  proof of the main theorems.  It
involves degeneration of surfaces to the boundary of the moduli space, 
and  then an analysis of the stable differentials. We show that the 
number of possible parameters of a periodic direction is finite. 
This last step towards finiteness is tackled using the Thurston-Veech's 
construction.
\par
In each step we complement the theoretical argument (implying the 
finiteness claim) by a way to implement this step in practice.
\par
We suspect that the stratum $\Prym[3](2,2)$ does not contain any
primitive \Teichmuller curve either. One way to prove this would be
to build the flat surfaces given by Thurston-Veech construction 
with the explicit tiles given in Section~\ref{sec:model22} and
tediously check that in a transverse direction the ratios of moduli are
not commensurable. This last check has been done successfully for one
of the eight topological models.
\par
\subsection*{Acknowledgments.} Computational assistance  was provided
by SAGE and a program written by Alex Eskin for computing cylinder 
decompositions. The authors are indebted to the programmers 
for the help provided. The first author is partially 
supported by the Labex Persyval and the Project GeoSpec.
The second author is partially supported by the DFG-project
``Classification of Teichm\"uller curves''.

\section{Prym loci and eigenforms}

\subsection{Prym loci} \label{sec:Prymloci}

In this context, a double covering $\pi:X \to Y$ of two Riemann surfaces is
called a {\em Prym covering}, if $g(X)-g(Y) =2$.  We refer to
the involution $\rho$ of $X$ with $Y = X/\langle \rho \rangle$ as the
{\em Prym involution}. This terminology is
taken from \cite{mcmullenprym} and differs from the classical terminology, 
where the polarization on the Prym variety ${\rm Prym}(X,\rho) = \Jac(X)/\pi^* \Jac(Y)$ induced from 
the principal polarization on $\Jac(X)$ had to be a multiple of a principal
polarization, but otherwise no genus restriction was imposed. 
\par
The {\em Prym locus} $\Prym[g](\kappa)$ is defined for any partition
$\kappa$ of $2g-2$ to be the set of flat surfaces $(X,\omega)$ 
such that $g(X)=g$, such that $X$ admits a Prym involution and
such that the zeros of $\omega$ are of type $\kappa$. For any $(X,\omega)$
in a Prym locus, the quadratic differential $q= \omega^2$ is $\rho$-invariant, 
hence a pull-back of from $Y$. Consequently, there is an 
$\SL_2(\RR)$-invariant isomorphism between the Prym loci and
strata of quadratic differentials on $Y$.
\par
Note that a zero of odd order is never fixed by the
Prym involution, but zeros of even order may be fixed or interchanged.
Consequently, there are two strata with $\kappa = (2,2)$.
If the two zeros are fixed, the Prym locus is isomorphic to
$Q(1,1,-1,-1)$ and part of the hyperelliptic component
of $\omoduli[3](2,2)^\hyp$. If the two zeros are interchanged, the
stratum Prym locus is isomorphic to $Q(4,-1^4)$ and part of
$\omoduli[3](2,2)^\odd$. Consequently, we denote these two
components by $\Prym[3](2,2)^\hyp$ and $\Prym[3](2,2)^\odd$ respectively.
Note that these two components are not of the same dimension.
\par
We give the full list of Prym loci in the case $g=3$ we are most
interested in. These are $\Prym[3](4)$ isomorphic to $Q(3,-1^3)$, 
the loci $\Prym[3](2,2)^\hyp$ and  $\Prym[3](2,2)^\odd$ mentioned above, 
the locus $\Prym[3](2,1,1)$ isomorphic to $Q(2,1,-1^3)$ and the
locus $\Prym[3](1,1,1,1)$ isomorphic to $Q(2^2,-1^4)$.
These five Prym loci in genus three are in fact connected, see
\cite{LNCompletePer} for the classification of 
connected components of Prym loci.

\subsection{Prym eigenform loci}

The main observation of \cite{mcmullenprym} was that the 
{\em Prym eigenform loci}, defined as intersection with 
the real multiplication locus
$$\Omega\cEE_D(\kappa)\, = \, \{(X,\omega) \in \Prym[g](\kappa)\, : \, 
{\rm Prym}(X,\rho) \, \text{has real multiplication by}\, \fraco_D \} $$
are $\SL_2(\RR)$-invariant for any discriminant~$D$ and a partition~$\kappa$ 
of $2g-2$ with $g \in \{2,3,4,5\}$.
\par
\Teichmuller curves in Prym loci are necessarily contained in the
Prym eigenform loci by \cite{moeller06}. We summarize when they coincide
and the current state of knowledge about these interesting curves.
\par
The intersection $\Prym[g](2g-2)$ 
with the minimal stratum consists of a union of \Teichmuller curves
for $g=2,3,4$ and this intersection is empty otherwise (\cite{mcmullenprym}). 
The connected components have been classified for $g=2$ in \cite{mcmullenspin} 
and in \cite{LNWeierstrass,LNWeierstrass2} for $g=3,4$. The topology of these curves 
is quite well-known. The Euler characteristic has been determined
in \cite{bainbridge07} for $g=2$ and in \cite{LNGLorbits} for $g=3,4$. The number of
cusps is calculated for $g=2$ in \cite{mcmullenspin} 
and in \cite{LNWeierstrass} for $g=3,4$. 
The elliptic elements are known for $g=2$ by work of
Mukamel (\cite{Mukamel}) and for $g=3,4$ by work of
Torres and Zachhuber (\cite{TZg3}, \cite{TZg4}).
\par
The eigenform loci $\Omega\cEE_D(2,2)$ are also two-dimensional, hence union
of \Teichmuller curves. Almost all of the obvious topological questions on these
curves are currently open.
\par
The situation is rather different in the cases when the dimension of the
eigenform locus in $\Prym[g](\kappa)$ is three or larger. The components
of the loci $\Prym[3](2,2)^\odd$ and $\Prym[3](2,1,1)$, which are of particular interest in this
paper, have been classified in \cite{LNEigComponents}. 
The same questions
(connected components, classification of \Teichmuller curves) naturally also
arise in $\Prym[3](1,1,1,1)$ and also in genus $4$ and $5$. They are currently 
open, mainly since the number of configurations of cylinders grows 
drastically with the dimension of the loci.
\par
Note that all surfaces in the eigenform loci that are given by 
the Thurston-Veech construction (see \cite[Section~4]{mcmullenprym} for
a concise survey) are completely algebraically periodic in the
sense of \cite{CaltaSmillie}. Hence all these surfaces have zero flux
and we cannot use the flux to rule out the existence of \Teichmuller
curves. Results on the more interesting property of
complete periodicity in Prym loci can be found in \cite{LNCompletePer}. 
However, this property is not relevant here either.

\section{Suitable degenerations}
\label{sec:degenerations}

It has been a recurrent theme in \cite{mcmullentor}, \cite{Mo08} , \cite{bainmoel}
that the constraints imposed by the torsion condition
of \cite{moellerper} are best expressible on a the stable curves 
over cusps of \Teichmuller curves. This applies in particular, 
if these stable curves are irreducible and rational curves with nodes.
\par
For primitive, but not algebraically primitive \Teichmuller curves
stable curves that are non-rational may appear at cusps. However, 
for the Prym loci under consideration we can easily find suitable cusps.
\par
The torsion condition translates into an equation in roots of unity
for any suitable cusp. In this section we derive this equation in roots of unity
both for $\Prym[3](2,2)$ and $\Prym[3](2,1,1)$.

\subsection{Stable curves,  stable differentials and the torsion condition}

Let $\barmoduli[g]$ denote the Deligne-Mumford compactification of the moduli space of
curves. A point in the boundary of $\barmoduli[g]$ is a {\em stable curve} $Y$, i.e.\  
a smooth curve $\widetilde{Y}$ with certain pairs of points $(x_i,y_i)$ identified to form 
nodes such that the automorphism group of $Y$ has finite order.
A {\em stable form} $\eta$ on $Y$ is a holomorphic 1-form on $\widetilde{Y} \, \setminus \, 
\bigcup_{i=1}^n (x_i,y_i)$ 
with at worst simple poles of opposite residues at each pair of points $(x_i,y_i)$.
\par
Given translation surface $(X,\omega)$ and a completely periodic direction $\Theta$, one can 
use the diagonal matrices $\left( \begin{smallmatrix} 1 & 0 \\ 0 & s \end{smallmatrix} \right)$ to 
push $(X,\omega)$ to the boundary of $\moduli[g]$. The limit point corresponds 
to a cusp of the Teichm\"uller curve $\overline{C} \setminus C$. Topologically the stable surface is obtained 
by collapsing the core curves of the cylinders to points, i.e.\ each cylinder gives
rise to a node and the irreducible components of the surface are obtained by cutting along the core curves. 
\par
This description also makes evident that the projective tuple of residues of $\eta$ equals
the tuple of widths of the cylinders.
\par
We call a direction $\Theta$ on a Veech surface and the corresponding cusp of the 
\Teichmuller curve {\em irreducible} if the stable curve is irreducible. We call
$\Theta$ and the cusp {\em suitable} if the stable curve is moreover a rational curve.
\par
We recall here the torsion condition that we want to exploit in our setting.
Suppose $(X,\omega)$ is a Veech surface and that $\omega$ has zeros $z_1,\ldots, z_m$.
Let $\Jac(X)^{-} = \Jac(X)/\Jac(X/\rho)$ be the $\rho$-anti-invariant part of the 
Jacobian of $X$.
\par
\begin{Thm}[\cite{moellerper}]
For any $i,j$  the divisor $[z_i-z_j]$ has finite order in $\Jac(X)^{-}$.
\end{Thm}
\par

\subsection{Suitable cusps in the case $\Prym[3](2,1,1)$}

We first show that directions of a saddle connection joining
the double zero to a simple zero are suitable for primitive \Teichmuller curves in this locus.
\par
\begin{Lemma}
\label{le:goodtopo211}
Let $(X,\omega) \in \Prym[3](2,1,1)$ be a Veech surface generating 
a primitive \Teichmuller curve. If $\Theta$ is the direction of a 
saddle connection joining the double zero to a simple zero, then the 
stable curve is an irreducible rational 
curve of geometric genus zero, i.e.\ the direction is suitable.
\par
The cylinder decomposition in the direction $\Theta$ consists of three 
cylinders, one fixed by $\rho$ and a pair that is exchanged by $\rho$.
\end{Lemma}
\par
\begin{Conv}
\label{conv:prym211}
{\rm We will always label the cylinders of a suitable direction such that 
$C_2$ is fixed by the Prym involution $\rho$, while the cylinders $C_1$ 
and $C_3$ are permuted.
\par
\noindent 
The widths and heights of the cylinders are denoted by $w_i = w(C_i)$ 
and $h_i = h(C_i)$.}
\end{Conv}
\par
\begin{proof} [Proof of Lemma~\ref{le:goodtopo211}]
Let $\Theta$ be a direction $\Theta$ on $(X,\omega)$ with a saddle connection joining
the double zero to a simple zero. Let $Y$ be the stable
surface obtained by degenerating using the geodesic flow in the
direction $\Theta$. One component of the stable curve thus contains
the double zero, a simple zero. This component is fixed by the Prym involution since
the double zero is fixed and hence it also contains the other simple zero. 
Consequently, $Y$ is an irreducible curve. Since $g(X) =3$, the curve $Y$ has at most three cusps
and the direction $\Theta$ at most three cylinders. Since $(X,\omega)$ is primitive the
widths of the core curve of the cylinders have $\QQ$-rank at least two. If two cylinders
were fixed, we would obtain a contradiction to the double zero being fixed. Hence
one cylinder is not fixed and this direction has precisely three cylinders. 
\par
This determines the genus of the normalization
of $Y$. It is isomorphic to a rational curve $\PP^1$ 
with three pairs of points that are identified.
\end{proof}
\par
Next, we express the stable form of a suitable direction algebraically.
We use $z$ as a coordinate on $\PP^1 \simeq \widetilde{Y}$. The Prym involution $\rho$
is preserved along the degeneration process, and induces an involution, still 
denoted by $\rho$, on $\PP^1$. We use the triple transitivity of M\"obius 
transformations to make the Prym involution 
$\rho$ on $\PP^1$ to be $z \mapsto -z$. Since the direction $\Theta$ decomposes 
the surface into three cylinders, it turns out that $\omega_\Theta$ has one 
double zero, two simple poles and three pairs of poles (with opposite residues). 
>From the action of $\rho$ on cylinders and 
zeros we deduce that one pair of poles is fixed while the other pairs are exchanged. 
Hence the pairs of identified points are $(u,-u)$, $(x,y)$ and $(-x,-y)$. 
The stable differential $\omega_\Theta$ is therefore given by,
\begin{equation} \label{eq:omegaasfraction}
\omega_\Theta \= C\cdot\frac{z^2(z^2-1)}{(z^2-x^2)(z^2-y^2)(z^2-u^2)}dz
\end{equation}
for some nonzero complex scalar $C$.
\par
Observe that Convention~\ref{conv:prym211} implies that the residue
of $\omega_\Theta$ at $x$ is $r_1$, 
and the residue at $u$ is $w_2$.  We can always rescale such that $r_1, r_2$ 
form a basis of a real quadratic field 
$K$. Stability of  $\omega_\Theta$ implies that the residue around $y$ is $-w_1$, 
namely
\begin{equation}
\label{eq:residue}
 \omega_\Theta \= r_1(\frac{1}{z-x} - \frac{1}{z-y})dz + r_2(\frac{1}{z-u} - 
\frac{1}{z+u})dz -  r_1(\frac{1}{z+x} - \frac{1}{z+y})dz.
\end{equation}
Hence using coefficient comparison between equations~\eqref{eq:residue} 
and~\eqref{eq:omegaasfraction},
we get the {\em opposite residue equations}
\begin{equation} \label{eq:oreqPrym211}
\begin{aligned}
0 & \= r_1u(y-x) +r_2xy \\ 
0 & \= r_1(yx^2 + (-y^2 + (-u^2 + 1))x + (u^2 - 1)y) - r_2u(x^2 + y^2 -1)\, ,
\end{aligned}
\end{equation}
where the first one is obtained from the constant term and the second one is 
obtained from the comparison of the $z^4$-term and $z^2$-terms. Note that 
the terms for odd powers of $z$ are automatically zero.
\par
\medskip
Finally, we express the torsion condition in these coordinates. The universal 
cover of $\Jac(X)^-$ is $\CC^2$, generated by $\omega$ and $\omega^\sigma$. 
The zeros $z_i$
of $\omega$ give rise to sections $Z_i$ of the universal family of curves over 
the \Teichmuller curve $C$. If $z_i - z_j$ is torsion for all fibers over $C$, then
the same holds for the corresponding Jacobian over the cusp. In the case of
a suitable degeneration this is the semiabelian variety
$$\CC^2/\langle{\rm Per}(\gamma_x), {\rm Per}(\gamma_u) \rangle 
\,\cong\, (\CC^*)^2, $$ 
where $\gamma_x$ and $\gamma_u$ are the loops around $x$ and $u$ respectively
and where ${\rm Per}$ denotes the vector of periods with respect to $\omega$ 
and $\omega^\sigma$.
The definition of the stable form in~\eqref{eq:residue} implies that 
the vector space of periods is
$$ {\rm Per}(\gamma_x) \= \ZZ\cdot\left(\begin{matrix} 2\pi I\cdot r_1
\\ 2\pi I\cdot r_1^\sigma \end{matrix} \right), \quad
{\rm Per}(\gamma_x) \= \ZZ\cdot\left(\begin{matrix} 2\pi I\cdot r_2
\\ 2\pi I\cdot r_2^\sigma \end{matrix} \right).
$$
On the other hand we calculate that the divisor given as the 
difference between the double zero a simple zero is the vector of
relative periods
\begin{equation} \label{eq:relper}
 \int_0^1 \omega_\Theta \= r_1 \log\left(\frac{(1-x)(1+y)}{(1+x)(1-y)}\right)
+ r_2 \log \left(\frac{1-u}{1+u}\right)
\end{equation} 
up to the contribution of closed paths, 
and similarly for $\omega^\sigma_\Theta$, replacing $r_i$ by $r_i^\sigma$.
\par
The torsion condition now amounts to 
\begin{equation} \label{eq:intvector}
\left(\begin{matrix} \int_0^1 \omega_\Theta \\
 \int_0^1 \omega^\sigma_\Theta \sigma 
\end{matrix} \right)
\,\in\, \QQ \cdot \left(\begin{matrix} 2\pi I\cdot r_1
\\ 2\pi I\cdot r_1^\sigma \end{matrix} \right) + \QQ \cdot 
\left(\begin{matrix} 2\pi I\cdot r_2
\\ 2\pi I\cdot r_2^\sigma \end{matrix} \right)
\end{equation}
\par
This is possible only if both logarithms in equation~\eqref{eq:relper}
lie in $ 2\pi I\cdot \QQ$. It will be convenient to make the invertible
change of variables
$$ X \= \frac{1-x}{1+x}, \quad Y \= \frac{1-y}{1+y}, 
\quad U \= \frac{1-u}{1+u}.$$
With this notation, the torsion condition amounts the existence of 
two roots of unity $\zeta_{XY}$ and $\zeta_U$, such that
\begin{equation} \label{eq:zetas}
Y/X \= \zeta_{XY} \quad \text{and} \quad U \= \zeta_U. 
\end{equation}
\par
We plug~\eqref{eq:zetas} solved for $Y$ in the first opposite residue 
equation~\eqref{eq:oreqPrym211} and clearing denominators the get the
first equation in~\eqref{eq:substcoeffroots}. We perform the same replacement
with the second opposite residue equation and subtract a suitable multiple
of the first equation to reduce the degree in $X$ from $4$ to three. The 
resulting equations are
\begin{equation} \label{eq:substcoeffroots}
\begin{aligned}
0 & \=  r_2(\zeta_U+1)\zeta_{XY}\, X^2  \\
& \+  (2 r_1 (\zeta_{XY}-1) (\zeta_U-1) \m r_2 (1 + \zeta_{XY}) (1 + \zeta_U))\, X  \\
& \+  r_2(\zeta_U+1) \quad\quad\quad\quad\quad \text{and} \\ 
0 & \=  \zeta_{XY}(16 \zeta_U  (\zeta_{XY} - 1) r_1 
\m  (\zeta^2_U - 1) (\zeta_{XY} + 1)r_2  )\, X^3  \\
& \+ (\zeta_U-1)(2(\zeta_U - 1) (\zeta_{XY}^2 - 1) r_1 
\m (\zeta_U + 1) (\zeta_{XY}^2 - 13\zeta_{XY} + 1)r_2)\,  X^2  \\
& \+ 2((\zeta_{XY} - 1) (\zeta_U^2 + 6\zeta_U + 1)r_1  
\m (\zeta^2_U - 1)(\zeta_{XY} + 1))X \\
& \+ r_2(\zeta_U^2 - 1).
\end{aligned}
\end{equation}
It turns out that the resultant of the two above equations
with respect to $X$, divided by the factor $256 (\zeta_U+1)\zeta_{XY}r_2$, is
 a square of the following equation 
\begin{equation} \label{eq:torsconseq211}
\begin{aligned}
0 &\= \r^2\zeta_{XY}\zeta_U^3 \m\r^2\zeta_{XY} \+ (4 -\r^2)\zeta_{XY}\zeta_U \+ 
(\r^2 -4)\zeta_{XY}\zeta_U^2  \\
& \+ (2 -\r)\zeta_{XY}^2\zeta_U^2 \+ (\r+2)\zeta_U^2 \m (\r+2) \zeta_{XY}^2\zeta_U 
\+ (\r -2) \zeta_U.
\end{aligned}
\end{equation}
With the application in the next section in mind, 
we have group the individual summands as powers of $\zeta_{XY}$ and $\zeta_U$. 
and we have scaled the projective tuple $(r_1,r_2)$ to $r_1 = 1$ and $r_2=r$ to improve 
readability.

\subsection{Suitable cusps in the case $\Prym[3](2,2)$}

With a proof similar to that of Lemma~\ref{le:goodtopo211} we obtain the following criterion
for suitable directions. 
\par
\begin{Lemma}
\label{le:goodtopo22}
Let $(X,\omega) \in \Prym[3](2,2)$ be a Veech surface generating a primitive \Teichmuller curve.
If $\Theta$ is the direction of a saddle connection joining
the double zeros  to a fixed point of the Prym involution, 
then the stable curve is an irreducible rational curve of geometric genus zero,
i.e.\ the direction is suitable.
\par
The cylinder decomposition in the direction $\Theta$ consists of 
three cylinders, one fixed by $\rho$ and a pair that is exchanged by $\rho$.
\end{Lemma}
\par
The saddle connection in the statement of the lemma is joining
the two double zeros, with a fixed point at its midpoint.
\par
\begin{proof}
In comparison with the proof of Lemma~\ref{le:goodtopo211} we
only have to rule out that the direction $\Theta$ has two cylinders, 
both fixed by $\rho$. Since then all four fixed points are contained
in the cylinders this contradicts the hypothesis on the saddle connection.
\end{proof}
\par
Consequently, we may stick to the labelling of cylinders as in Convention~\ref{conv:prym211} here, too.
We also continue to use $\rho$ to be $z \mapsto -z$ on the normalization $\widetilde{Y} \cong \PP^1$
of any suitable cusp. However, now both zeros are fixed by the Prym involution, so that
$$
\omega_\Theta \= C \cdot \frac{(z-1)^2(z+1)^2}{(z^2-x^2)(z^2-y^2)(z^2-u^2)}\, dz.
$$
Comparing with
$$
\omega_\Theta \= r_1(\frac{1}{z-x} - \frac{1}{z-y})dz + r_2(\frac{1}{z-u} - 
\frac{1}{z+u})dz -  r_1(\frac{1}{z+x} - \frac{1}{z+y})dz.
$$
we obtain the {\em opposite residue equations}
\begin{equation} \label{eq:oreq22}
\begin{aligned} 
0 & \= (x - y) (xyu^2 + 1)r_1 - u  (xy - 1) (xy + 1)r_2\, , \\
0 & \= (y-x )  (-xy + u^2 - 2)r_1 - u (x^2 + y^2 - 2)r_2. 
\end{aligned}
\end{equation}
\par
\medskip
The relative period is now given as the integral between the two zeros of $\omega$
\begin{equation} \label{eq:relper22}
 \int_{-1}^1 \omega_\Theta \= 2r_1 \log\left(\frac{(1-x)(1+y)}{(1+x)(1-y)}\right)
+ 2r_2 \log \left(\frac{1-u}{1+u}\right)
\end{equation} 
(up to the contribution of closed paths).
The torsion condition implies the existence 
of two roots of unity $\zeta_{XY}$ and $\zeta_{U}$ such that
$$
\zeta_U=U \qquad \textrm{and} \qquad  \zeta_{XY} = Y/X, \textrm{ where } U=\frac{1-u}{1+u}, \,
X = \frac{1-x}{1+x},\
Y = \frac{1-y}{1+y}.    
$$
Similarly to the situation in $\Prym[3](2,1,1)$ we plug this (solved for $X$) into the opposite residue
equations and clear denominators.  We obtain two expression of degree two in $X$. A suitable linear
combination of the two is $64X$ times the following relation, that we again specialized to $r_1 =1$, and $r_2=r$.
\begin{flalign}
\label{eq:Prym22}
&r(r-1) \zeta_{XY}\ +\ r(r-1)\zeta_{XY}^2\zeta_{U}^4\ -\ 2(r^2-1)\zeta_{XY}\zeta_{U}^2\ -\ 2(r^2-1)\zeta_{XY}^2\zeta_{U}^2& \nonumber\\
 &+\ r(r+1)\zeta_{XY}^2\ +\ r(r+1)\zeta_{XY}\zeta_{U}^4\ -\ 2\zeta_{XY}^3\zeta_{U}^2\ -\ 2\zeta_{U}^2 \= 0.
\end{flalign}
\subsection{Degenerate cases of Equation~\eqref{eq:torsconseq211}}
\label{rq:trivial:solution}
We analyze degenerate cases of Equation~\eqref{eq:torsconseq211} in order to eliminate
them from the discussion in the next section.  Solutions with $\zeta_{XY} =1$ or $\zeta_U=1$ are of no interest
for us since then $x=y$ (respectively, $u=0$) which would
imply that two poles (respectively a pole and a zero) of the stable come together.
Idem for $\zeta_U=-1$.

\subsection{Double coverings of ${\mathcal Q}(1,1,-1^6)$} \label{sec:Q11-16}
Similarly we discuss the degenerate cases of Equation~\eqref{eq:Prym22}. If $\zeta_U = \pm 1$ then
the poles corresponding to $\pm u$ have collided or they have collided with
the location $z = \pm 1$ of a zero. If $\zeta_{XY} = 1$ then the poles 
$x$ and $y$ have collided. We can exclude these cases.
\par
If  $(\zeta_{XY},\zeta_{U})=(-1,\pm \imath)$ then Equation~\eqref{eq:Prym22} 
holds for any $r$. In this situation, $x=y^{-1}$ and $u=\pm \imath$. Then the 
involution $h: z \mapsto 1/z$  fixes all the cylinders and the zeros.
It has thus $8$ fixed points, i.e.\ $h$ is a limiting case of 
the hyperelliptic involution and the quotient belongs to 
$\mathcal Q(1,1,-1^6)$. Since all the cylinders are fixed, hyperelliptic
open-up argument \cite[Proposition~3.4]{chenmoeller} applies and we conclude that
any Teichm\"uller curve limiting to such a boundary point is a family
of hyperelliptic curves. Since the hyperelliptic involution is unique, 
$h$ and $\rho$ commute, they generate a $(\ZZ/2)^2$, whose remaining
involution we denote by $\tau$. Since $3 =g(X/\rho) + g(X/\tau) + g(X/h)$
we conclude that $X/\tau$ is a genus two curve. If $(X,\omega)$ is
a generating Veech surface, then $\tau^* \omega = h^* \, \rho^* \omega = \omega$, 
hence $\omega$ is a pullback from $X/\tau$.
\par
To summarize, boundary points with  $(\zeta_{XY},\zeta_{U})=(-1,\pm \imath)$
lie only on \Teichmuller curves generated by imprimitive Veech surfaces
and need not be considered for the proof of Theorem~\ref{thm:main}.

\section{Solving linear relations in roots of unity.}

Equations~\eqref{eq:torsconseq211} and~\eqref{eq:Prym22} are instances
of what Mann (\cite{Mann}) calls linear relations in roots of unity. More 
precisely, a {\em $K$-relation of length~$k$} is an equation 
\begin{equation*}
\sum_{i=1}^{k} a_i \zeta_i \= 0 
\end{equation*}
where the $\zeta_i$ are pairwise different roots of unity and where 
all $a_i$ lie in the number field $K$. The relation is called {\em irreducible},
if $\sum_{i=1}^k b_i \zeta_i =0$ and $b_i(a_i-b_i)=0$ for
all $i$ implies that $b_i=0$ for all $i$ or $a_i-b_i=0$ for all $i$.
Obviously each relation is a sum of irreducible relations, but there may be
several ways of writing a relation as sum of irreducible relations.
\par
Mann proved a finiteness statement for $\QQ$-relations in roots of unity. This 
was generalized in \cite{Mo08} to coefficients in a number field of bounded 
degree. We use here the following version with better bounds of 
Dvornicich-Zannier on the possible orders of roots 
of unity. We state the special case of $[K:\QQ]=2$.
\par
\begin{Thm}[{\cite[Theorem 1]{dvorzanniersum}}] \label{Thm:DvZa}
Suppose that $\sum_{i=1}^{k} a_i \zeta_i = 0$ is an irreducible $K$-relation. 
Then there exists a primitive $N$-th root unity $\zeta_N$, some root of 
unity $\xi$ and exponents $b_i \in \ZZ$ such that $\zeta_i = \xi \zeta_N^{b_i}$, 
where $N$ is bounded as follows. 
\par
The exponents of primes dividing $N$ are bounded by the condition that if 
$p^{\alpha+1} | N$  for some prime $p \in \ZZ$ then $p^\alpha\, |\, 2d$ where 
$d:=[K:\QQ]$. Moreover the size of primes dividing $N$ is
bounded by
$$\sum_{p|N} \left(\frac{p-1}{(p-1,d)} -1 \right)\leq k - 2\,. $$
\end{Thm}
\par
The purpose of this section is to provide algorithms with feasible bounds to
find all solutions to Equations~\eqref{eq:torsconseq211} and~\eqref{eq:Prym22} that are
relevant from the point of view of cusps of \Teichmuller curves. 
We denote $\cNN(k)$ the set of orders that may appear for a $K$-relation of length $k$
according to Theorem~\ref{Thm:DvZa}. Note that $\cNN(k)$ is monotone in $k$, i.e.\ if
$\ell \geq k$ then $\cNN(\ell) \subset \cNN(k)$.

\subsection{The $K$-relation for $\Prym[3](2,1,1)$}
We will prove
\begin{Thm}
\label{thm:technical211}
There are at most finitely many solutions of Equation~\eqref{eq:torsconseq211} 
\begin{equation*}
\begin{array}{lllllllll}
\r^2\zeta_{XY}\zeta_U^3 &-&\r^2\zeta_{XY} &+& (4 -\r^2)\zeta_{XY}\zeta_U &+& (\r^2 -4)\zeta_{XY}\zeta_U^2 & + \\
(2 -\r)\zeta_{XY}^2\zeta_U^2 &+& (\r+2)\zeta_U^2 &-& (\r+2) \zeta_{XY}^2\zeta_U &+& (\r -2) \zeta_U &= 0.
\end{array}
\end{equation*}
where $\r$ is real, $[\QQ(\r):\QQ] =2 $ and where $\zeta_{XY},\zeta_{U}$ are roots of unity
with  $\zeta_{XY} \neq 1$ and $\zeta_U \not \in \{\pm1\}$. All the solutions 
are given by $N$-th roots of unity for some $N \in \cNN(8)$.
\par 
More precisely,  there are only $16$ such solutions
with $N^K_\QQ(\r) < 0$. They are  given in the 
following table 
but only one representative from each pair 
$(\zeta_{XY},\zeta_U,\r) = (\zeta_N^{e_{XY}},\zeta_N^{e_{U}},\r)$
and $(\zeta_{XY},\zeta_U,\r) = (\zeta^{N-e_{XY}}_N,\zeta^{N-e_{U}}_N,r)$.
\begin{table}[htbp]
$$
\begin{array}{|c|c|c|c|}
\hline
N & e_{XY} & e_{U} & \r  \\
\hline
6 & 1 & 1 & \tfrac{3+\sqrt{33}}{6} \\
6 &3 & 1 & \tfrac{2\sqrt{6}}{3} \\
6 &3 & 2 & 2\sqrt{2} \\
6 & 5 & 1 & \tfrac{-3+\sqrt{33}}{6}  \\
\hline
12 & 6 & 1 & {-2+2\sqrt{3}} \\
12 & 6 & 5 & {\phantom{-}2+2\sqrt{3}}  \\
\hline
24 & 3 & 4 & \tfrac{2\sqrt{3}}{3} \\
24 & 15 & 4 & \tfrac{2\sqrt{3}}{3} \\
\hline
\end{array}
$$
\end{table}
\end{Thm}
\par
The reason for restricting to the solutions with $N^K_\QQ(\r) < 0$ will become
clear in Section~\ref{sec:heights}. Note that since $r_2$ is real, the solutions
come in complex conjugate pairs. 
\par
The idea of proof is to combine Theorem~\ref{Thm:DvZa} with the fact that 
the roots of unity appearing in equation~\eqref{eq:torsconseq211} are not 
arbitrary but rather products of only two roots of unity $\zeta_{XY}$ 
and~$\zeta_U$.  More precisely we will apply Theorem~\ref{Thm:DvZa} to all possible ways 
of writing \eqref{eq:torsconseq211} as sums of irreducible $K$-relations. The following
condition is helpful. Suppose an irreducible relation contained in~\eqref{eq:torsconseq211} 
contains at least three terms
$a_j\zeta_{XY}^{\alpha_j} \zeta_U^{\beta_j}$ for $j \in J$ with $|J| \geq 3$.
We call such an irreducible relation {\em good}, 
if there are three indices $i,j,k \in J$, such that 
$$d_{ijk} \,:=\, \det\left(\begin{matrix} \alpha_i-\alpha_j & \alpha_i - \alpha_k \\
\beta_i - \beta_j& \beta_i - \beta_k \\ \end{matrix} \right)$$
is non-zero.
\par
\begin{Lemma} \label{le:dgcd}
Suppose that an irreducible relation $\sum_{j \in J} a_j\zeta_{XY}^{\alpha_j} \zeta_U^{\beta_j} =0$ 
contained in Equation~\eqref{eq:torsconseq211} is good and let 
$$d = \gcd_{i,j,k \in J} \{d_{ijk}\}.$$ 
Then the order of both $\zeta_{XY}$ and $\zeta_U$ is an element of $d\cNN(|J|)$.
\end{Lemma}
\par
\begin{proof}
By Theorem~\ref{Thm:DvZa} we know that there exists an $N$-th root of unity 
$\zeta_N$ with $N$ in  $\cNN(|J|)$
and some root of unity $\xi$ such that $ \zeta_{XY}^{\alpha_j} \zeta_U^{\beta_j} = \xi \zeta_N^{a(j)}$ 
for some integer $a(j)$. For any good triple of indices $(i,j,k)$ we can solve this for
$\zeta_{XY}$ and $\zeta_U$ being a root of unity in $d_{ijk}\cNN(|J|)$. Considering this
condition jointly for all such
triples we conclude that they are in fact roots of unity in  the set $d\cNN(|J|)$.
\end{proof}
\par
We will distinguish two cases, $K$-relations of 
length two and longer $K$-relations.
$K$-relations of length two are never good, but they impose other strong constraints 
and we deal with them separately. Labelling of the terms is in the order 
they appear in Equation~\eqref{eq:torsconseq211}.
\par
\begin{Lemma} \label{lm:relation:length2}
If a solution to the Equation~\eqref{eq:torsconseq211} 
has a sub-$K$-relation of length two,
then at least one of the following statements holds true.
\begin{enumerate}
\item[i)] The relation consists of terms (12) and $\zeta_U^3 =1$.
\item[ii)] The relation consists of terms (58) and $\zeta_U = \zeta_{XY}^{-2}$.
\item[iii)]  The relation consists of terms (67) and $\zeta_U = \zeta_{XY}^{2}$.
\end{enumerate}
Moreover, in this case all the solutions
of  Equation~\eqref{eq:torsconseq211} are of the form $(\zeta_{XY}, \zeta_U) = 
(\zeta_N^{e_{XY}}, \zeta_N^{e_U})$ 
for some $N \in \cNN(5)$.
\end{Lemma}
\par
\begin{proof}
If an irreducible relation contained in~\eqref{eq:torsconseq211} has precisely two terms
\begin{equation} \label{eq:K2conseq}
a_1\zeta_{XY}^{\alpha_1} \zeta_U^{\beta_1}+ a_2\zeta_{XY}^{\alpha_2} \zeta_U^{\beta_2} \= 0,
\quad \text{then} \quad 
-a_1/a_2\=\zeta_{XY}^{\alpha_2-\alpha_1} \zeta_U^{\beta_2-\beta_1} \in \{\pm 1\}
\end{equation}
since these are the only real roots of unity. 
\par
Suppose that the expressions for $a_1$ and $a_2$ as they appear 
in~\eqref{eq:torsconseq211} are not equal up to sign. If the $a_i$ 
are both  linear in $\r$ this gives a solution 
$\r \in \QQ$, contradiction. We run into the same contradiction for 
$a_1 = \pm (\r^2-4)$ and $a_2$ linear (and vice versa), since
the linear terms that appear divide this $a_1$. For $a_1 = \pm \r^2$  
and $a_2$ linear (and vice versa), the solution is either non-real or in $\QQ$, 
contradiction. The case of both~$a_i$ being quadratic in~$\r$  is
possible (given that $\zeta_U \neq \pm 1$) only 
if the relation is (13) or (24). In either case we deduce
$\zeta_U^2 = -1$ and $\r = \sqrt{2}$. Then we can 
solve~\eqref{eq:torsconseq211} for $\zeta_{XY}$, in fact  
$\zeta_{XY}^2 = \tfrac13(-1 \pm  2i\sqrt{2})$, which is not solvable
for a root of unity.
\par
If the expressions for $a_1$ and $a_2$ as they appear 
in~\eqref{eq:torsconseq211}
agree up to sign, we are in one of the cases listed, since the relation (34)
implies $\zeta_U = \pm 1$. We discuss the resulting equations separately.
\par
In case i) we obtain after taking resultant with $\zeta_U^3 -1 =0$ and dividing 
by $\zeta_{XY}^2-1$the relation
$$ (\r^2 + 12)\zeta_{XY}^4 + (12\r^2 - 48)\zeta_{XY}^3 + (3\r^4 - 26\r^2 + 72)\zeta_{XY}^2 
+ (12\r^2 - 48)\zeta_{XY} + \r^2 + 12 = 0.$$
In case ii) we obtain after substituting and dividing by $\zeta_{XY}^2-1$
$$\r^2\zeta_{XY}^4 + (\r + 2)\zeta_{XY}^3 + (2\r^2 - 4)\zeta_{XY}^2 
+ (\r + 2)\zeta_{XY} + \r^2 = 0.$$
The relation stemming from the last case iii) is the same as in the
previous case after replacing $\r$ by $-\r$.
\par
If none of the coefficients is zero in any of the two preceding 
displayed equations, these are either irreducible (hence with a 
solution for $\zeta_{XY}$ in for $N \in \cNN(5)$) or
reducible with a partition in two plus three terms, hence with a 
solution for $\zeta_{XY}$ 
in for $N \in \cNN(3)$. Since $\zeta_{U}$ is determined, being a power of $\zeta_{XY}$, 
or a third root of unity, the claim follows in this case. 
\par
None of the coefficients of the first equation is annihilated by 
a real quadratic irrational number and $\r =  \tfrac{\sqrt{2}}2$
is the only interesting case for the second equation. 
Some of the solutions lie on the unit circle, but a loop over the
finitely many roots of unity of degree $16$ over $\QQ$ shows that these
solutions are not roots of unity. 
\end{proof}
\par
\begin{proof}[Proof of Theorem~\ref{thm:technical211}, finiteness statement]
The only triples that are not good are (labeling of the terms as they appear 
in Equation~\eqref{eq:torsconseq211}) $(4,5,6)$, 
$(3,7,8)$ all the triples in $\{1,2,3,4\}$. 
By the preceding lemma we only need to deal with atom-free partitions of $(8)$
without a part of length two since the number of possibilities for 
$N \in \cNN(5)$ is finite. The only possibilities are $(5,3)$, $(4,4)$ and $(8)$. 
The relation of length $8$ is obviously good. By the preceding statement
none of the relations of length $5$ is bad. If the relation of
length $4$ is the bad one, consisting of $(1234)$, then the complementary
relation is good. Consequently, all 
these situations are good and Lemma~\ref{le:dgcd} gives the finiteness statement.
\end{proof}

\subsection{Implementation for the case $\Prym[3](2,1,1)$}

The goal is to prove the second part of Theorem~\ref{thm:technical211}, 
namely to show that all the solutions are given by the table
in Theorem~\ref{thm:technical211}. 
The first step is to reduce the discussion to the partition $(8)$ into just 
one irreducible relation, since we have to
treat this case anyway, and the second step is to feasibly implement this case.
\par
We start with the second step and deal with {\bf Case~$(8)$}.
We have $\cNN(8) = 2^3\cdot3 \cdot \{17, 5\cdot 13, 7\cdot 11, 5\cdot 11, 5\cdot 7\}$.
For each order $N$ in that list and for each pair of exponents $(e_X,e_U)$
we have to try whether for $\zeta_{XY} = \zeta_N^{e_X}$ and $\zeta_{U} = \zeta_N^{e_U}$
the quadratic equation (in $\r$) with coefficients in that cyclotomic field
has a solution in a quadratic number field. \medskip
\par
Equation~\eqref{eq:torsconseq211} can be rewritten as $a\r^2 +  b\r + c = 0$, where
$$
\left\{\begin{array}{lcl}
a&=&\zeta_N^{e_X} (\zeta_N^{e_U}-1)(\zeta_N^{e_U}+1)^2 \\
b&=&-\zeta_N^{e_U} (\zeta_N^{e_U}+1) (\zeta_N^{e_X}-1)(\zeta_N^{e_X}+1)\\
c&=&2 \zeta_N^{e_U} (\zeta_N^{e_U}-1) (\zeta_N^{e_X}-1)^2
\end{array}\right.
$$
Observe that $a\neq 0$ since $\zeta_U \neq \pm 1$. We can solve the above 
quadratic equation as follows. We let $\Delta=\sqrt{b^2-4ac}$.
\begin{enumerate}
\item[i)] If $b/a \in \QQ$ and $c/a \in \QQ$ then $[\QQ(r):\QQ] \leq 2$
and we record the value, if the degree is two.
\item[ii)] Otherwise, if $b/a \in \QQ$ but $\sqrt{\Delta}/a \not \in 
\QQ(\zeta_N)$, then $\sqrt{\Delta}/a$ belongs to a quadratic extension of $\QQ$, 
linearly disjoint from $\QQ(\zeta_N)$. But then $\Delta/a^2 \in \QQ$, hence 
$c/a \in \QQ$, contradiction. If $b/a \not \in \QQ$ and $\sqrt{\Delta}/a \not \in 
\QQ(\zeta_N)$, then $2r = b/a \pm  \sqrt{\Delta}/a$ cannot have degree two over $\QQ$.
\item[iii)] If $\sqrt{\Delta}/a \in \QQ(\zeta_N)$, compute $r$ as element of
$\QQ(\zeta_N)$ and record the value, if the degree over $\QQ$ is two.
\end{enumerate}
\par
For practical purposes the loop over $N^2$ cases can be simplified significantly.
First, we may suppose that ${\rm gcd}(e_x,e_U,N)=1$, since otherwise the case
has been treated before, for a divisor of $N$. Suppose first that
we are in the 'rational' case i). Then for any  $\sigma \in \Gal(\QQ(\zeta_N)/\QQ)$
we need to find common solutions of $ab^\sigma - a^\sigma b =0$ and  $ac^\sigma - a^\sigma c =0$.
Taking resultants, this reduces to a loop over just one variable.
\par
Suppose now that $r \in \QQ(\zeta_N)$ as in case iii). Then $\Gal(\QQ(\zeta_N)/\QQ)$
acts on the set of solutions of~\eqref{eq:torsconseq211} without preserving the condition
$N(\r)<0$. Using this action we may suppose that $e_X | N$, thus reducing the loop
to $N$ times the number of divisors of $N$ cases.
\par
Testing if  $\Delta$ has a square root in $\QQ(\zeta_N)$ in step iii) is 
a time-consuming operation. In practice it works much better to assume that
$r_2$ satisfies a quadratic equation $r_2^2 + \beta r_2 + \gamma = 0$, to take
the resultant with~\eqref{eq:torsconseq211} and solve the resulting system
for rational number $\beta, \gamma$. If this system does not determine  
$\beta, \gamma$ up to finitely many choice, we can still fall back on the 
original test.
\par
\medskip
We now discuss step one.  Suppose that one of the relations, say of length~$k$, 
is good and $d=1$. Then we find all solutions by a loop over 
all $N \in \cNN(k)$, all $e_X, e_U$ and $\zeta_{XY} = \zeta_N^{e_X}$ 
and $\zeta_{U} = \zeta_N^{e_U}$ and we need to check
if Equation~\eqref{eq:torsconseq211} holds for these values form some $\r$, 
quadratic over $\QQ$. Since $\cNN(\cdot)$ is monotone in the argument, 
these cases have already been checked.
Also, by Lemma~\ref{lm:relation:length2} all the partition with a 
sub-$K$-relation of length two have already been checked.
\par
\medskip
\noindent{\bf Case (3,5).} For every such partition, one of the two relations has $d=1$, 
so all the cases have been dealt with.
\par
\noindent {\bf Case (4,4).} The only 4-tuples for which $d\neq 1$ are listed below, 
with their complementary relation:
$$
\begin{array}{|c|ccc|}
\hline
& \textrm{relation} & \textrm{complementary relation} &\textrm{values of d} \\
\hline
(1) & (1, 3, 5, 6) &   (2, 4, 7, 8) &2\,\, \text{resp.}\,\, 2 \\
(2) &(1, 3, 7, 8) &   (2, 4, 5, 6) &2\,\, \text{resp.}\,\, 2\\
(3) &(1, 2, 3, 4) &   (5, 6, 7, 8) &0\,\, \text{resp.}\,\, 2\\
\hline
\end{array}
$$
In each case we consider the resultant of the two equations with respect to $\zeta_{XY}$.
In the first two cases this factors completely and we obtain, respectively,
$$
16 (r - 2)  (r - 1)  (r + 1)  (r + 2)  (\zeta_U - 1)^2  (\zeta_U + 1)^2  \zeta_U^4 \= 0
$$
and
$$
- (r - 2)  (r + 2)  \zeta_U^2  (\zeta_U - 1)^4  (\zeta_U + 1)^4  r^4 \= 0.
$$
Since $\zeta_U \neq \pm 1$ and $r\not \in \QQ$ there is no solution. 
\par
Finally, 
for the pair of relations in (3) the resultant with respect to $\zeta_{XY}$ gives:
$$
\zeta_U  (\zeta_U - 1)^2 \cdot \left( r(\zeta_U + 1) + 2\zeta_U - 2 \right)\cdot  \left( r^2(\zeta_U^2 + 2\zeta_U + 1) - 4\zeta_U \right)^2 \= 0.
$$
If $r(\zeta_U + 1) + 2\zeta_U - 2 = 0$ then the equation given by terms $(5, 6, 7, 8)$ 
becomes
$\zeta_U  (\zeta_U - 1)  (\zeta_U + 1)  \zeta_{XY}^2=0$ leading to a contradiction.
The remaining possibility is $r^2(\zeta_U^2 + 2\zeta_U + 1) - 4\zeta_U = 0$. 
This is a length 3 irreducible relation:
$$
r^2\zeta_U^2 + (2r^2 -4)\zeta_U + r^2 = 0.
$$
If $\r \neq \sqrt{2}$ the solutions for $\zeta_U$ correspond to $N$-the roots 
of unity for $N\in \cNN(3)$.  Now using the resultant with respect to $\zeta_U$ 
we find (after dividing by rational factors) a cubic equation for $\zeta_{XY}$
which is irreducible unless $\r = \sqrt{4/3}$ and $\zeta_{XY}$ is an $8$-the
root of unity. In any case, such a solution  has 
appeared while searching the irreducible case of Equation~\eqref{eq:torsconseq211}.
\par
If $\r = \sqrt{2}$ then the same resultant shows that 
$\zeta_{XY}^2 = \frac{-3\pm \sqrt{-7}}4$. This number lies on the unit circle, but this
is not a root of unity. 
\par
\medskip
\medskip

\subsection{The $K$-relation for $\Prym[3](2,2)$}
We now discus the relation in roots of unity in the other stratum.
\par
\begin{Thm}
\label{thm:technical22}
There are at most finitely many solutions of Equation~\eqref{eq:Prym22}
\begin{equation*}
\begin{aligned}
&\r(\r-1) \zeta_{XY}\ +\ \r(\r-1)\zeta_{XY}^2\zeta_{U}^4\ -\ 2(\r^2-1)\zeta_{XY}\zeta_{U}^2\ -\ 2(\r^2-1)\zeta_{XY}^2\zeta_{U}^2&\\
 &+\ \r(\r+1)\zeta_{XY}^2\ +\ \r(\r+1)\zeta_{XY}\zeta_{U}^4\ -\ 2\zeta_{XY}^3\zeta_{U}^2\ -\ 2\zeta_{U}^2 = 0.
\end{aligned}
\end{equation*}
where $\r$ is real, $[\QQ(\r):\QQ] =2 $ and where $\zeta_{XY},\zeta_{U}$ are roots of unity
with  $\zeta_{XY} \neq 1$,  $\zeta_U \not \in \{\pm1\}$
and $(\zeta_{XY},\zeta_{U}) \neq (-1,\pm \imath)$. All the solutions 
are given by $N$-th roots of unity for some $N \in \cNN(8) \cup 2\cdot \cNN(4)\cup 4\cdot \cNN(4)$.
\par
More precisely, there are only $32$ such solutions
with $N^K_\QQ(\r) < 0$. They are given in the following table 
but only one representative from each pair 
$(\zeta_{XY},\zeta_U,\r) = (\zeta_N^{e_{XY}},\zeta_N^{e_{U}},\r)$
and $(\zeta_{XY},\zeta_U,\r) = (\zeta^{N-e_{XY}}_N,\zeta^{N-e_{U}}_N,r)$.
\begin{table}[htbp]
$$
\begin{array}{lllll}
\begin{array}{|c|c|c|c|}
\hline
N&e_{XY} & e_{U} & r \\
\hline
12&2 & 3 & (\sqrt{2})/2 \\
12&10 & 3 & (\sqrt{2})/2 \\
\hline
12&1 & 3 & (-1+\sqrt{3})/2 \\
12&1 & 9 & (-1+\sqrt{3})/2 \\
12&5 & 3 & (1+\sqrt{3})/2 \\
12&7 & 3 & (1+\sqrt{3})/2 \\
\hline
12&4 & 3 &  (\sqrt{6})/2 \\
12&4 & 9 &  (\sqrt{6})/2 \\
\hline
\end{array} &&
\begin{array}{|c|c|c|c|}
\hline
N&e_{XY} & e_{U} & r \\
\hline
12&4 & 1 & (3+\sqrt{33})/2 \\
12&4 & 5 & (-3+\sqrt{33})/2 \\
12&8 & 1 & (-3+\sqrt{33})/2 \\
12&8 & 5 & (3+\sqrt{33})/2 \\
\hline
48&16 & 21 & \sqrt{3} \\
48&16 & 9 & \sqrt{3} \\
48&32 & 3 & \sqrt{3} \\
48&32 & 15 & \sqrt{3} \\
\hline
\end{array}
\end{array}
$$
\end{table}
\end{Thm}
\par
The proof is completely parallel to Theorem~\ref{thm:technical211}.
\par
\begin{Lemma} \label{lm:relation22:length2}
If a solution to the Equation~\eqref{eq:Prym22} 
has a sub-$K$-relation of length two, then the relation consists
of terms $(12)$, $(56)$ or $(78)$. Moreover, in this case all the solutions
of   Equation~\eqref{eq:Prym22} are of the form $(\zeta_{XY}, \zeta_U) = 
(\zeta_N^{e_{XY}}, \zeta_N^{e_U})$ 
for some $N \in \cNN(5)$.
\end{Lemma}
\begin{proof}
As in Lemma~\ref{lm:relation:length2} we use that 
\begin{equation} \label{eq:K2conseq22}
a_1\zeta_{XY}^{\alpha_1} \zeta_U^{\beta_1}+ a_2\zeta_{XY}^{\alpha_2} \zeta_U^{\beta_2} = 0,
\quad \text{implies} \quad 
-a_1/a_2=\zeta_{XY}^{\alpha_2-\alpha_1} \zeta_U^{\beta_2-\beta_1} \in \{\pm 1\}
\end{equation}
Only the polynomials $\{ -2,r(r-1),  r(r+1),-2(r^2-1)\}$ appear as
coefficients of the equation~\eqref{eq:Prym22}.
If $a_1 \neq a_2$ then solving~\eqref{eq:K2conseq22} for $r$ gives a solution 
in $\QQ$, except when $a_1=-2(r^2-1)$ and $a_2=-2$ 
(or vice versa). In this case $-a_1/a_2 = -(r^2-1)=-1$ and $r=\sqrt{2}$. 
>From terms (37) and (48) we deduce the equations
$\zeta_{XY}\zeta_{U}^2 + \zeta_{XY}^3\zeta_{U}^2=0$ and 
$\zeta_{XY}^2\zeta_{U}^2+\zeta_{U}^2=0$ respectively.  In these  two cases $\zeta_{XY}^2=-1$
and the equation reduces to $3\zeta_U^8 + 2\zeta_U^4 + 3=0$, which is possible
only if  $\tfrac{-1\pm 2\imath\sqrt{2}}{3}$ were a root of unity. A check
over the finitely many roots of unity that are quadratic shows that this is 
impossible. The remaining combinations (38) and (47) imply 
$\zeta_{XY}=-1$ and then also $\zeta_U = \pm i$, 
contradicting the hypothesis of Theorem~\ref{thm:technical22}.
\par
Hence it remains to treat the case where $a_1=a_2$. The partition (34) implies that
$\zeta_{XY}=-1$, contradiction. We discuss the remaining cases, the partitions
$(12)$, $(56)$ and $(78)$. Each of the remaining terms has a factor 
$(\zeta_U^2+1)(\zeta_U^2-1)$. If $\zeta_U^2 = -1$ then
Equation~\eqref{eq:Prym22} reduces to 
$$ \zeta_{XY}^3 + (2\r^2-1)\zeta_{XY}^2 + (2\r^2-1)\zeta_{XY} + 1\=0.$$
If $\r = \sqrt{2}/2$ then $\zeta_{XY}^3 = -1$, a solution that appears for $N=6$, 
see the table in Theorem~\ref{thm:technical22}. 
Otherwise, none of the coefficients of this 
relation of length three is zero. Hence the relation is irreducible, any
solution appears for $N \in \cNN(3)$. We analyze the remaining cases separately. They are
\begin{enumerate}
\item $\zeta_{XY} + \zeta_{XY}^2\zeta_{U}^4=0$, 
\item $\zeta_{XY}^2 + \zeta_{XY}\zeta_{U}^4 = 0$, 
\item $\zeta_{XY}^3\zeta_{U}^2 + \zeta_{U}^2=0$.
\end{enumerate}
In case $(1)$ Equation~\eqref{eq:Prym22} becomes, after taking out 
the factors we discussed,
$$
2\zeta_U^8 + r(r+1)\zeta_U^6 - 2(r^2-2)\zeta_U^4 + r(r+1)\zeta_U^2 + 2 \= 0.
$$
In case $(2)$ Equation~\eqref{eq:Prym22} is the preceding equation, with $\r$
replaced by $-\r$. Finally in case $(3)$, Equation~\eqref{eq:Prym22} becomes, 
after taking out the factors we discussed,
$$
r^2(3r^2+1)\zeta_U^8 - 12r^2(r^2-1)\zeta_U^6 + 2r^2(9r^2-13)\zeta_U^4 
- 12r^2(r^2-1)\zeta_U^2 + r^2(3r^2 + 1) \= 0
$$
If none of the coefficients is zero in any of the two preceding displayed equations, 
these are either irreducible (hence with a solution for $\zeta_U$ in for $N \in \cNN(5)$) or
reducible with a partition in two plus three terms, hence with a solution for $\zeta_U$ 
in for $N \in \cNN(3)$. Since $\zeta_{XY}$ is determined, being a power of $\zeta_U$, the
claim follows in this case. 
\par
The only cases of a real quadratic number annihilating one of the coefficients 
are $\r = \tfrac{\sqrt{2}}2$, $\r = \tfrac{\sqrt{3}}3$ 
and $\r = \tfrac{\sqrt{13}}3$. In each
case some of the solutions lie on the unit circle, but a loop over the
finitely many roots of unity of degree $16$ over $\QQ$ shows that these
solutions are not roots of unity. 
\end{proof}
\par
\begin{proof}[Proof of Theorem~\ref{thm:technical22}, finiteness statement]
The only triples (and four-tuples) that are not good are (labeling of the 
terms as they appear in~\eqref{eq:Prym22}) $(1,3,6)$, $(2,4,5)$ and all the 
triples contained in ${3,4,7,8}$ as well as the four-tuple $(3,4,7,8)$. 
By the preceding lemma we only need to deal with atom-free partitions of $(8)$
without a part of length two since the number of possibilities for 
$N \in \cNN(5)$ is finite. The only possibilities are $(5,3)$, $(4,4)$ and $(8)$. 
The relation of length $8$ is obviously good and by the preceding statement, 
if one of the relations of length $3$ resp.\ $4$ is bad, the complementary
relation of length $5$ or $3$ is good.
\end{proof}

\subsection{Implementation for the case $\Prym[3](2,2)$}

The basic algorithm of the preceding case applies here as well. We first
test all the possible choices for $(\zeta_{XY},\zeta_U)$ corresponding to 
an irreducible relation. Equation~\eqref{eq:Prym22} can also be written as
a quadratic polynomial in $\r$, the coefficients being polynomial
in $\zeta_{XY}$ and $\zeta_U$. Consequently, all the preceding remarks on whose
to efficiently test the cases apply here as well.
\par
The reduction step to the irreducible case works less well here, compared
to the case $\Prym[3](2,1,1)$. In fact, there are $56$ partitions 
into a triple and a relation of length $5$ where the
greatest common divisor of the two $d$ associated by Lemma~\ref{le:dgcd} 
with the triple and the $5$-tuple is equal to two. Moreover, there
are $64$ partitions into two four-tuples such that the corresponding 
greatest common divisor is two and, finally, there are $6$ such partitions 
where the greatest common divisor is equal four. 
\par
Not all of them can be ruled out as we did in Case~(4,4) for $\Prym[3](2,1,1)$, 
besides the fact that this is not feasible, since the solutions with
$N =48$ stem from such cases. Instead, we loop over all 
$N \in \cNN(8) \cup 2\cdot \cNN(4)\cup 4\cdot \cNN(4)$, as stated in
the theorem to cover all the cases.

\section{Finiteness of the input data for the Thurston-Veech construction}

In this section we complete the proof of the finiteness statements
in Theorem~\ref{thm:main} and Theorem~\ref{thm:main:2}.
First, we determine the heights of the cylinders of any suitable direction
and justify the norm condition we restricted ourselves to when tabulating
the solutions to the relations in roots of unity. 
The second statement is a finiteness result for relative periods. The third step 
is to argue via the Thurston-Veech construction that finiteness holds.

\subsection{Heights} \label{sec:heights}

The heights of the cylinders in a suitable direction are determined  
in the loci $\Prym[3](2,1,1)$ and $\Prym[3](2,2)$ by the widths, 
real multiplication and the Prym involution.
\par
\begin{Lemma} \label{le:heights}
Normalizing $r_1=1$ and writing $r_2 = a + b\sqrt{D}$, 
the tuple heights of the cylinders is proportional to 
$(h_1 = 1,h_2,h_3=h_1)$, where
$$
h_2 \= \frac{2a+2b\sqrt{D}}{Db^2-a^2}.
$$
\end{Lemma}
\par
\begin{proof} We use the condition on the complex flux implied by complete
periodicity, which states (\cite[Theorem~4.5]{mcmulleninfinite}, see 
also \cite{CaltaSmillie}) that $\sum_{i=1}^k r_ih^\sigma_i=0$ for any 
direction that completely decomposes into $k$ cylinders.
In our setting, this amounts to $2r_1 h_1^\sigma + r_2h_2^\sigma = 0$
and this gives the above value.
\end{proof}
\par
\begin{Cor} \label{cor:normneg}
For any tuple $(r_1=1, r_2)$ of widths of cylinders
in a suitable direction on a Veech surface in $\Prym[3](2,1,1)$
or $\Prym[3](2,2)$ we have  $N^K_\QQ(r_2) <0$.
\end{Cor}
\par
\begin{proof} Since  $h_1/h_2 = -r_2^\sigma/2>0$
and since $r_2>0$ we deduce $N^K_\QQ(r_2) <0$.
\end{proof}
\par

\subsection{Relative periods}

A saddle connection $\gamma$ is said to {\em represent a relative period},
if it joins a simple zero to the double zero in the locus $\Prym[3](2,1,1)$
or if it joins the two double zeros in the locus $\Prym[3](2,2)$. Note
that this is an abuse of terminology, since a saddle connection
joining the two simple zeros does not qualify.
\begin{Prop} 
\label{prop:irrallfinite}
There exists a finite set $\mathcal R \subset \PP(\RR^3)$ 
such that for any Veech surface 
$(X,\omega)\in\Prym[3](2,1,1)$ or in $\Prym[3](2,2)$, 
any suitable direction $\Theta$ on $X$
and any saddle connection $\gamma$ representing a relative period
the tuple $(r_1:r_2:|\gamma|)$ belongs to $\mathcal R$.
\end{Prop}
\par
\begin{proof}
We start with the case $\Prym[3](2,1,1)$.
First, the tuple $(r_1=1,r_2, \zeta_{XY}, \zeta_U)$ is constrained
in any suitable direction  by Equation~\eqref{eq:torsconseq211}, 
to which there are only finitely many solutions
by Theorem~\ref{thm:technical211}. We may  assume 
that the torsion order $N$ 
and  one of the finitely many  choices for $\zeta_{XY} = \zeta_N^{e_X}$ 
and $\zeta_U = \zeta_N^{e_U}$ as in the table in 
Theorem~\ref{thm:technical211} 
are fixed. Using Equation~\eqref{eq:substcoeffroots}, we see that this 
tuple also
determines $X,Y,U$ and hence $\omega_\Theta$ up to 
finite ambiguity. 
\par
Let $\delta$ be the path along the real axis from $z=0$ to $z=1$. 
Direct computation gives:
\begin{equation}
\begin{aligned}
 \int_\delta \omega_\Theta &\= r_1 \left[ \log \frac{z-x}{z+x} -  \log \frac{z-y}{z+y} \right]_0^1  + 
r_2 \left[\frac{z-u}{z+u} \right]_0^1 \\
& \= r_1 \log(\zeta^{-1}_{XY})  +r_2 \log(\zeta_U) = -r_1 \frac{e_X}{N} + r_2 \frac{e_U}{N}. \\
\end{aligned}
\end{equation}
\par 
The saddle connection $\gamma$ is a simple curve homotopic to $\delta$ up to a union 
of loops once around  a subset of the points 
$\{x,y,-x,-y,u,-u\}$. The residue of $\omega_\Theta$ around these points 
is $\pm r_1$ and $\pm r_2$ respectively. Since $\gamma$ is disjoint
from $\rho(\gamma)$ up to the endpoint, we do not need a loop both
around $x$ and $-x$ etc. We obtain
\begin{equation} \label{eq:relperfromdelta}
\int_\gamma \omega_\Theta  \= \int_\delta \omega_\Theta + k r_1 + \ell r_2, \quad
k \in \{-2,-1,0,1,2\}, \quad \ell \in \{-1,0,1\}.
\end{equation}
This is a finite list of possibilities.
\par
In the case $\Prym[3](2,2)$ similarly the knowledge of the roots of
unity given as the solutions of the equation in Theorem~\ref{thm:technical22}
determines the stable form by the calculation that lead to 
Equation~\eqref{eq:Prym22}.
\par
The computation of the relative period (as already given in~\eqref{eq:relper22})
yields 
\begin{equation} 
\begin{aligned}
 \int_{-1}^1 \omega_\Theta &\= 2r_1 \log\left(\frac{(1-x)(1+y)}{(1+x)(1-y)}\right)
+ 2r_2 \log \left(\frac{1-u}{1+u}\right) \\
& = 2r_1 \log(\zeta^{-1}_{XY})  + 2r_2 \log(\zeta_U) = -2r_1 \frac{e_X}{N} + 2r_2 \frac{e_U}{N}\, . 
\end{aligned}
\end{equation} 
The same argument as above, based on the observation that $\delta$
is a simple curve gives here again that Equation~\eqref{eq:relperfromdelta}
holds.
\end{proof}

\subsection{Finiteness of the possibilities for the intersection matrix}

We call a pair $(\Theta_1, \Theta_2)$ of suitable directions on
a Veech surface $(X,\omega) \in \Prym[3](2,1,1)$ 
{\em admissible}, if there exists a saddle connection $\beta$ in
the direction $\Theta_2$ that represents a relative period and
that crosses one of the cylinders of the direction $\Theta_1$ just once
and intersects no other cylinder of the direction $\Theta_1$.
Similarly, we call a Veech surface $(X,\omega)$ in $\Prym[3](2,2)$
{\em admissible}, if  $(\Theta_1, \Theta_2)$ are as above but now
we ask that $\beta$ crosses the fixed cylinder (called $C_2$) once
and crosses no other cylinder {\em or} that  $\beta$ crosses both
exchanged cylinders (called $C_1$ and $C_3$) once and crosses no other 
cylinder.
\par
We call a Veech surface $(X,\omega) \in \Prym[3](2,1,1)$ or 
in $\Prym[3](2,2)$ {\em normalized}, if the horizontal and
vertical direction are an admissible pair of suitable directions
and if the cylinder $C_1$ (according to Convention~\ref{conv:prym211})
has width $r_1=1$ and $h_1=1$.
\par
\begin{Lemma}
Any Veech surface $(X,\omega) \in \Prym[3](2,1,1)$ or 
in $\Prym[3](2,2)$ can be normalized.
\end{Lemma}
\par
\begin{proof} A suitable direction $\Theta_1$ exists by
Lemma~\ref{le:goodtopo211} resp.\ Lemma~\ref{le:goodtopo22}. 
Now it suffices to take a cylinder
in the direction $\Theta_1$ with zero of different orders
(resp.\ with two different double zeros) on its boundaries.
Take $\beta$ to be the saddle connection joining these zeros
and $\Theta_2$ to be the direction of $\beta$. Obviously the
pair $(\Theta_1, \Theta_2)$ is admissible. The normalization
is an obvious consequence of the transitivity of the action 
of $\GL_2(\RR)$.
\end{proof}
\par
We denote by $Z_j$, $j=1,2,3$ the cylinders of the vertical direction, 
also labeled according to Convention~\ref{conv:prym211}. The goal of
this section is the following finiteness statement for
the intersection matrix.
\par
\begin{Prop} 
\label{prop:intmatfinite}
Suppose that $(X,\omega)$ is normalized.
There is only a finite number of possibilities for the heights
and widths $w(C_i), h(C_i), w(Z_i), h(Z_i)$ of the cylinders
in the horizontal and vertical direction as well as for 
the intersection matrix $M = (M_{ij})$ where $M_{ij} = C_i\cdot Z_j$ 
is the geometric intersection
number of horizontal and vertical cylinders.
\end{Prop}
\par
As preparation, recall that the obvious properties
$$
w(Z_j) \= \sum_{i=1}^3 M_{ij} h(C_i) \qquad \textrm{and} \qquad 
w(C_i) \= \sum_{j=1}^3 M_{ij} h(Z_j)
$$
are expressed in terms of the matrix $M = (M_{ij})$ as 
$$
w(Z)\=M^T\cdot h(C) \qquad \textrm{and} \qquad w(C)\=M\cdot h(Z).
$$
The point here is that $M$ is never invertible. But we can 
defined a reduced intersection matrix, that turns out to be invertible.
We define the reduced intersection matrix to be 
$M^{\rm red}  = \left( \begin{smallmatrix} M_{11}+M_{13} & 2M_{12} \\
M_{21} & M_{22} \\ \end{smallmatrix} \right)$,
and 
since $M_{11} = M_{33}$, $M_{13} = M_{31}$, $M_{21}  = M_{23}$ and $M_{12} = M_{32}$,
the preceding equation reads
\begin{equation}
\label{eq:galois:intersection:matrix}
w(Z)^{\rm red}\=(M^{\rm red})^T\cdot h(C)^{\rm red}, \qquad \textrm{where }
\left( \begin{smallmatrix} a \\ b \\a \end{smallmatrix}\right)^{\rm \red} \= 
\left( \begin{smallmatrix} a \\ b \end{smallmatrix}\right).
\end{equation}
\par
\begin{proof}[Proof of Proposition~\ref{prop:intmatfinite}]
Given a horizontal and a vertical irreducible direction, we 
may fix one of the finitely many choices for $r_2$ given by Proposition~\ref{prop:irrallfinite} and the heights are given by Lemma~\ref{le:heights}.
The projective tuple $(w(Z_1) : w(Z_2) : |\beta|)$
is one of the finitely many tuples appearing according to
Proposition~\ref{prop:irrallfinite}. Since $|\beta| = 1$ or
$|\beta| = h_2$ according to the definition of an admissible direction, 
it suffices to prove finiteness for a fixed 
tuple  $(w(Z_1),  w(Z_2),  |\beta|)$. By
 Lemma~\ref{le:heights} again this determines also the $h(Z_i)$.
\par
\medskip
We now use that the intersection numbers are integral, in particular
invariant under Galois conjugation. The pairs $(w(C_1), w(C_2)$
and hence also $(w(Z_1), w(Z_2)$ are a $\QQ$-basis of $K$.
The same applies to the vertical direction. Consequently, the system
\begin{equation} \label{eq:computeMred}
\left(\begin{matrix}
w(Z_1) & w(Z_1)^\sigma \\
w(Z_2) & w(Z_2)^\sigma \\
\end{matrix}
\right) \= 
(M^{\rm red})^T\cdot
\left(\begin{matrix}
h(C_1) & h(C_1)^{\sigma} \\
h(C_2) & h(C_2)^\sigma \\
\end{matrix}
\right).
\end{equation}
can be solved uniquely
for the matrix $M^{\rm \red}$, for any of the fixed possibilities 
for $(h_1,h_2,w(Z_1),w(Z_2))$. 
By integrality and positivity, 
there is only a finite number of possibilities for $M_{11}$
and $M_{13}$ given $M_{11}+M_{13}$. Together with the conditions stated
at the beginning of the proof, this determines the intersection
matrix completely.
\end{proof}
\par
\begin{proof}[Proof of Theorem~\ref{thm:main} and 
Theorem~\ref{thm:main:2}, finiteness statement]
We recall that by the Thurston-Veech construction 
(see e.g.\ \cite{mcmullenprym})
a Veech surface with periodic horizontal and vertical directions
of given heights $h(C_i)$ and $h(Z_j)$ is composed of rectangles
of size $h(C_i) \times h(Z_j)$. The number of such rectangles is
bounded by the total sum of the entries of the intersection matrix.
All these quantities are finite by Proposition~\ref{prop:intmatfinite}.
\end{proof}

\subsection{Implementation}
\label{sec:implementation}
The finiteness proof above is constructive and can be implemented
by performing in the case $\Prym[3](2,1,1)$ for each of the 
fields $\QQ(\sqrt{D_0})$ where\footnote{Note that so far we have not
been specifying the order of real multiplication, just the field and
we specify it by  a square-free integer~$D_0$.}  
$D_0\in\{2,3,6,33\}$ as appearing in table in 
Theorem~\ref{thm:technical211} 
resp.\ in the case $\Prym[3](2,2)$  for the $D_0\in \{2,3,6,33\}$  
as they appear in table in Theorem~\ref{thm:technical22} 
the following steps.
\begin{enumerate}
\item \label{s2} Compute the complete geometry in each suitable direction, 
i.e.\ for every solution according to the table 
with the given~$D_0$ store the possible triples
$(r_1=1,r_2,\gamma)$ in a list $\mathcal R$, where 
$$
\quad \quad \quad |\gamma| =  - \frac{e_X}{N} + r_2 \frac{e_U}{N} + k + \ell r_2, \quad \text{for} \quad 
k \in \{-2,-1,0,1,2\}, \quad \ell \in \{-1,0,1\}
$$
for  $\Prym[3](2,1,1)$ and
$$
\quad \quad \quad |\gamma| =  - 2\frac{e_X}{N} + 2r_2 \frac{e_U}{N} + k + \ell r_2, \quad \text{for} \quad 
k \in \{-2,-1,0,1,2\}, \quad \ell \in \{-1,0,1\}
$$
for  $\Prym[3](2,2)$ respectively. Moreover, $\gamma$ 
satisfies $0< |\gamma| <\max\{1,r_2\}$, since we
may assume that $\gamma$ is a saddle connection on the boundary of
one of the cylinders. Store also $(h_1,h_2)$ computed 
according to Lemma~\ref{le:heights}.
\medskip
\item For any pair of tuples in $\mathcal R$ normalize the first element of
the pair to be equal to 
$(r_1=1, r_2, h_1=1, h_2)$ (forgetting about the saddle connection) 
and normalize the second tuple 
to be  $(w(Z_1), w(Z_2), |\beta|, h(Z_1),h(Z_2))$ with 
either  
$$
|\beta| \= h_1 \qquad \textrm{or} \qquad |\beta| \= h_2
$$
corresponding to the cases  if $|\beta|$ crosses $C_1$ or $C_2$ 
for  $\Prym[3](2,1,1)$ and
$$
|\beta| \= 2h_1 \qquad \textrm{or} \qquad |\beta| \= h_2
$$
corresponding to the cases  if $|\beta|$ crosses $C_1$ or $C_2$ 
for  $\Prym[3](2,2)$ respectively. Compute $M^\red$ according to \eqref{eq:computeMred} and store
the matrix if it is non-negative integral and $(M^{\rm red})_{12}\equiv 0 \mod 2$.
\end{enumerate}
\par
\smallskip
The list of cases for $\Prym[3](2,1,1)$ is reduced even more by the 
following constraint, that is easily checked using the complete list 
of configurations given in Figure~\ref{fig:SuitableDir211}.
\par
\begin{Lemma} \label{le:211crossC1}
Suppose that $(X,\omega) \in \Prym[3](2,1,1)$ has a horizontal direction, 
normalized according to Convention~\ref{conv:prym211}. Then there is 
a suitable direction with a saddle connection $\beta$ contained 
in $C_1$ crossing this cylinder once.
\end{Lemma}
\par 
The corresponding statement for $\Prym[3](2,2)$ is even easier to
obtain. Note that $C_2$, being fixed by the Prym involution, contains
a fixed point of the  Prym involution in its interior. This 
proves the following lemma. 
\par
\begin{Lemma} \label{le:22crossC1}
Suppose that $(X,\omega) \in \Prym[3](2,2)$ has a horizontal direction, 
normalized according to Convention~\ref{conv:prym211}. Then there is 
a suitable direction with a saddle connection $\beta$ contained 
in $C_2$ crossing this cylinder once.
\end{Lemma}

\begin{figure}[htbp]
\begin{center}
\begin{tikzpicture}[scale=0.8, every node/.style={transform shape}]


\draw (0,25) rectangle (4,26) node [midway] {$C_1$}; 
\draw (5,25) rectangle (6,26) node [midway] {$C_2$};
\draw (7,25) rectangle (11,26) node [midway] {$C_3$};
\draw (0,25) -- (1,25) node [midway, below] {$5$};
\draw (1,25) -- (2,25) node [midway, below] {$4$};
\draw (2,25) -- (3,25) node [midway, below] {$7$};
\draw (3,25) -- (4,25) node [midway, below] {$6$};
\draw (5,25) -- (6,25) node [midway, below] {$1$};
\draw (7,25) -- (10,25) node [midway, below] {$2$};
\draw (10,25) -- (11,25) node [midway, below] {$3$};
\draw (0,26) -- (1,26) node [midway, above] {$1$};
\draw (1,26) -- (4,26) node [midway, above] {$2$};
\draw (5,26) -- (6,26) node [midway, above] {$3$};
\draw (7,26) -- (8,26) node [midway, above] {$4$};
\draw (8,26) -- (9,26) node [midway, above] {$5$};
\draw (9,26) -- (10,26) node [midway, above] {$6$};
\draw (10,26) -- (11,26) node [midway, above] {$7$};
\draw[fill=black] (5,25) circle (3pt) (6,25) circle (3pt)
                  (7,25) circle (3pt) (10,25) circle (3pt)
                  (11,25) circle (3pt) (0,26) circle (3pt)
                  (1,26) circle (3pt) (4,26) circle (3pt)
                  (5,26) circle (3pt) (6,26) circle (3pt);
\draw (0,25) circle (3pt) (2,25) circle (3pt)
      (4,25) circle (3pt) (8,26) circle (3pt)
      (10,26) circle (3pt);
\draw (1,25) node {$\times$} (3,25) node {$\times$}
      (7,26) node {$\times$} (9,26) node {$\times$}
      (11,26) node {$\times$};

\draw (0,22.5) rectangle (4,23.5) node [midway] {$C_1$}; 
\draw (5,22.5) rectangle (6,23.5) node [midway] {$C_2$};
\draw (7,22.5) rectangle (11,23.5) node [midway] {$C_3$};
\draw (0,22.5) -- (1,22.5) node [midway, below] {$5$};
\draw (1,22.5) -- (2,22.5) node [midway, below] {$1$};
\draw (2,22.5) -- (3,22.5) node [midway, below] {$7$};
\draw (3,22.5) -- (4,22.5) node [midway, below] {$3$};
\draw (5,22.5) -- (6,22.5) node [midway, below] {$4$};
\draw (7,22.5) -- (10,22.5) node [midway, below] {$2$};
\draw (10,22.5) -- (11,22.5) node [midway, below] {$6$};
\draw (0,23.5) -- (1,23.5) node [midway, above] {$1$};
\draw (1,23.5) -- (4,23.5) node [midway, above] {$2$};
\draw (5,23.5) -- (6,23.5) node [midway, above] {$3$};
\draw (7,23.5) -- (8,23.5) node [midway, above] {$4$};
\draw (8,23.5) -- (9,23.5) node [midway, above] {$5$};
\draw (9,23.5) -- (10,23.5) node [midway, above] {$6$};
\draw (10,23.5) -- (11,23.5) node [midway, above] {$7$};
\draw[fill=black] (0,22.5) circle (3pt) (3,22.5) circle (3pt)
                  (4,22.5) circle (3pt) (5,22.5) circle (3pt)
                  (6,22.5) circle (3pt) (5,23.5) circle (3pt)
                  (6,23.5) circle (3pt) (7,23.5) circle (3pt)
                  (8,23.5) circle (3pt) (11,23.5) circle (3pt);
\draw (2,22.5) circle (3pt) (7,22.5) circle (3pt)
      (11,22.5) circle (3pt) (1,23.5) circle (3pt)
      (10,23.5) circle (3pt);
\draw (1,22.5) node {$\times$} (10,22.5) node {$\times$}
      (0,23.5) node {$\times$} (9,23.5) node {$\times$};

\draw (0,20) rectangle (3,21) node [midway] {$C_1$}; 
\draw (5,20) rectangle (6,21) node [midway] {$C_2$};
\draw (7,20) rectangle (10,21) node [midway] {$C_3$};
\draw (0,20) -- (1,20) node [midway, below] {$2$};
\draw (1,20) -- (2,20) node [midway, below] {$6$};
\draw (2,20) -- (3,20) node [midway, below] {$3$};
\draw (5,20) -- (6,20) node [midway, below] {$1$};
\draw (7,20) -- (8,20) node [midway, below] {$5$};
\draw (8,20) -- (9,20) node [midway, below] {$7$};
\draw (9,20) -- (10,20) node [midway, below] {$4$};
\draw (0,21) -- (1,21) node [midway, above] {$1$};
\draw (1,21) -- (2,21) node [midway, above] {$2$};
\draw (2,21) -- (3,21) node [midway, above] {$3$};
\draw (5,21) -- (6,21) node [midway, above] {$4$};
\draw (7,21) -- (8,21) node [midway, above] {$5$};
\draw (8,21) -- (9,21) node [midway, above] {$6$};
\draw (9,21) -- (10,21) node [midway, above] {$7$};
\draw[fill=black] (1,20) circle (3pt) (2,20) circle (3pt)
                  (8,20) circle (3pt) (2,21) circle (3pt)
                  (8,21) circle (3pt) (9,21) circle (3pt);;
\draw (7,20) circle (3pt) (9,20) circle (3pt)
      (10,20) circle (3pt) (5,21) circle (3pt)
      (6,21) circle (3pt) (7,21) circle (3pt)
      (10,21) circle (3pt);
\draw (0,20) node {$\times$} (3,20) node {$\times$}
      (5,20) node {$\times$} (6,20) node {$\times$}
      (0,21) node {$\times$} (1,21) node {$\times$} 
      (3,21) node {$\times$};

\draw (0,17.5) rectangle (3,18.5) node [midway] {$C_1$}; 
\draw (4,17.5) rectangle (7,18.5) node [midway] {$C_2$};
\draw (8,17.5) rectangle (11,18.5) node [midway] {$C_3$};
\draw (0,17.5) -- (1,17.5) node [midway, below] {$2$};
\draw (1,17.5) -- (2,17.5) node [midway, below] {$6$};
\draw (2,17.5) -- (3,17.5) node [midway, below] {$3$};
\draw (4,17.5) -- (6,17.5) node [midway, below] {$1$};
\draw (6,17.5) -- (7,17.5) node [midway, below] {$5$};
\draw (8,17.5) -- (9,17.5) node [midway, below] {$7$};
\draw (9,17.5) -- (11,17.5) node [midway, below] {$4$};

\draw (0,18.5) -- (2,18.5) node [midway, above] {$1$};
\draw (2,18.5) -- (3,18.5) node [midway, above] {$2$};
\draw (4,18.5) -- (5,18.5) node [midway, above] {$3$};
\draw (5,18.5) -- (7,18.5) node [midway, above] {$4$};
\draw (8,18.5) -- (9,18.5) node [midway, above] {$5$};
\draw (9,18.5) -- (10,18.5) node [midway, above] {$6$};
\draw (10,18.5) -- (11,18.5) node [midway, above] {$7$};
\draw[fill=black] (0,17.5) circle (3pt) (3,17.5) circle (3pt)
                  (6,17.5) circle (3pt) (9,17.5) circle (3pt)
                  (2,18.5) circle (3pt) (5,18.5) circle (3pt)
		  (8,18.5) circle (3pt) (11,18.5) circle (3pt);
\draw (2,17.5) circle (3pt) (8,17.5) circle (3pt)
      (11,17.5) circle (3pt) (4,18.5) circle (3pt)
      (7,18.5) circle (3pt) (10,18.5) circle (3pt);
\draw (1,17.5) node {$\times$} (4,17.5) node {$\times$}
      (7,17.5) node {$\times$} (0,18.5) node {$\times$}
      (3,18.5) node {$\times$} (9,18.5) node {$\times$};

\draw (0,15) rectangle (3,16) node [midway] {$C_1$}; 
\draw (4,15) rectangle (7,16) node [midway] {$C_2$};
\draw (8,15) rectangle (11,16) node [midway] {$C_3$};
\draw (0,15) -- (1,15) node [midway, below] {$7$};
\draw (1,15) -- (2,15) node [midway, below] {$2$};
\draw (2,15) -- (3,15) node [midway, below] {$3$};
\draw (4,15) -- (6,15) node [midway, below] {$1$};
\draw (6,15) -- (7,15) node [midway, below] {$5$};
\draw (8,15) -- (9,15) node [midway, below] {$6$};
\draw (9,15) -- (11,15) node [midway, below] {$4$};

\draw (0,16) -- (2,16) node [midway, above] {$1$};
\draw (2,16) -- (3,16) node [midway, above] {$2$};
\draw (4,16) -- (5,16) node [midway, above] {$3$};
\draw (5,16) -- (7,16) node [midway, above] {$4$};
\draw (8,16) -- (9,16) node [midway, above] {$5$};
\draw (9,16) -- (10,16) node [midway, above] {$6$};
\draw (10,16) -- (11,16) node [midway, above] {$7$};
\draw[fill=black] (2,15) circle (3pt) (4,15) circle (3pt)
                  (7,15) circle (3pt) (8,15) circle (3pt)
                  (11,15) circle (3pt) (0,16) circle (3pt)
		  (3,16) circle (3pt) (4,16) circle (3pt)
                  (7,16) circle (3pt) (9,16) circle (3pt);
\draw (0,15) circle (3pt) (3,15) circle (3pt)
      (9,15) circle (3pt) (5,16) circle (3pt)
      (10,16) circle (3pt);
\draw (1,15) node {$\times$} (6,15) node {$\times$}
      (2,16) node {$\times$} (8,16) node {$\times$}
      (11,16) node {$\times$};

\draw (0,12.5) rectangle (3,13.5) node [midway] {$C_1$}; 
\draw (4,12.5) rectangle (9,13.5) node [midway] {$C_2$};
\draw (10,12.5) rectangle (13,13.5) node [midway] {$C_3$};
\draw (1,12.5) -- (2,12.5) node [midway, below] {$6$};
\draw (4,12.5) -- (7,12.5) node [midway, below] {$1$};
\draw (7,12.5) -- (8,12.5) node [midway, below] {$5$};
\draw (8,12.5) -- (9,12.5) node [midway, below] {$7$};
\draw (10,12.5) -- (13,12.5) node [midway, below] {$4$};

\draw (0,13.5) -- (3,13.5) node [midway, above] {$1$};
\draw (4,13.5) -- (5,13.5) node [midway, above] {$2$};
\draw (5,13.5) -- (6,13.5) node [midway, above] {$3$};
\draw (6,13.5) -- (9,13.5) node [midway, above] {$4$};
\draw (10,13.5) -- (11,13.5) node [midway, above] {$5$};
\draw (11,13.5) -- (12,13.5) node [midway, above] {$6$};
\draw (12,13.5) -- (13,13.5) node [midway, above] {$7$};
\draw[fill=black] (1,12.5) circle (3pt) (2,12.5) circle (3pt)
                  (8,12.5) circle (3pt) (5,13.5) circle (3pt)
                  (11,13.5) circle (3pt) (12,13.5) circle (3pt);
\draw (0,12.5) circle (3pt) (3,12.5) circle (3pt)
      (10,12.5) circle (3pt) (13,12.5) circle (3pt)
      (4,13.5) circle (3pt) (6,13.5) circle (3pt)
      (9,13.5) circle (3pt);
\draw (4,12.5) node {$\times$} (7,12.5) node {$\times$}
      (0,13.5) node {$\times$} (3,13.5) node {$\times$}
      (10,13.5) node {$\times$} (13,13.5) node {$\times$};

\draw (0,10) rectangle (2,11) node [midway] {$C_1$}; 
\draw (4,10) rectangle (7,11) node [midway] {$C_2$};
\draw (8,10) rectangle (10,11) node [midway] {$C_3$};
\draw (0,10) -- (1,10) node [midway, below] {$2$};
\draw (1,10) -- (2,10) node [midway, below] {$3$};
\draw (4,10) -- (5,10) node [midway, below] {$5$};
\draw (5,10) -- (6,10) node [midway, below] {$1$};
\draw (6,10) -- (7,10) node [midway, below] {$6$};
\draw (8,10) -- (9,10) node [midway, below] {$7$};
\draw (9,10) -- (10,10) node [midway, below] {$4$};
\draw (0,11) -- (1,11) node [midway, above] {$1$};
\draw (1,11) -- (2,11) node [midway, above] {$2$};
\draw (4,11) -- (5,11) node [midway, above] {$3$};
\draw (5,11) -- (6,11) node [midway, above] {$4$};
\draw (6,11) -- (7,11) node [midway, above] {$5$};
\draw (8,11) -- (9,11) node [midway, above] {$6$};
\draw (9,11) -- (10,11) node [midway, above] {$7$};
\draw[fill=black] (0,10) circle (3pt) (2,10) circle (3pt)
                  (6,10) circle (3pt) (9,10) circle (3pt)
                  (1,11) circle (3pt) (5,11) circle (3pt)
                  (8,11) circle (3pt) (10,11) circle (3pt);
\draw (1,10) circle (3pt) (5,10) circle (3pt)
      (0,11) circle (3pt) (2,11) circle (3pt)
      (4,11) circle (3pt) (7,11) circle (3pt);
\draw (4,10) node {$\times$} (7,10) node {$\times$}
      (8,10) node {$\times$} (10,10) node {$\times$}
      (6,11) node {$\times$} (9,11) node {$\times$};

\draw (0,7.5) rectangle (2,8.5) node [midway] {$C_1$}; 
\draw (0,7.5) -- (1,7.5) node [midway, below] {$7$};
\draw (1,7.5) -- (2,7.5) node [midway, below] {$2$};
\draw (0,8.5) -- (2,8.5) node [midway, above] {$1$};
\draw[fill=black] (0,8.5) circle (3pt) (2,8.5) circle (3pt);
\draw (0,7.5) node {$\times$} (2,7.5) node {$\times$};
\draw (1,7.5) circle (3pt);
\draw (3,7.5) rectangle (8,8.5) node [midway] {$C_2$};
\draw (3,7.5) -- (4,7.5) node [midway, below] {$3$};
\draw (4,7.5) -- (6,7.5) node [midway, below] {$1$};
\draw (6,7.5) -- (7,7.5) node [midway, below] {$5$};
\draw (7,7.5) -- (8,7.5) node [midway, below] {$6$};
\draw (3,7.5) node {$\times$} (8,7.5) node {$\times$} (4,8.5) node {$\times$};
\draw[fill=black] (4,7.5) circle (3pt) (6,7.5) circle (3pt) 
                  (7,8.5) circle (3pt) (5,8.5) circle (3pt);
\draw (3,8.5) circle (3pt) (8,8.5) circle (3pt) (7,7.5) circle (3pt);
\draw (3,8.5) -- (4,8.5) node [midway, above] {$2$};
\draw (4,8.5) -- (5,8.5) node [midway, above] {$3$};
\draw (5,8.5) -- (7,8.5) node [midway, above] {$4$};
\draw (7,8.5) -- (8,8.5) node [midway, above] {$5$};
\draw (9,7.5) rectangle (11,8.5) node [midway] {$C_3$};
\draw (9,7.5) -- (11,7.5) node [midway, below] {$4$};
\draw (9,8.5) -- (10,8.5) node [midway, above] {$6$};
\draw (10,8.5) -- (11,8.5) node [midway, above] {$7$};
\draw[fill=black] (9,7.5) circle (3pt) (11,7.5) circle (3pt);
\draw (9,8.5) circle (3pt) (11,8.5) circle (3pt);
\draw (10,8.5) node {$\times$};

\draw (0,5) rectangle (1,6) node [midway] {$C_1$};
\draw (0,5) -- (1,5) node [midway, below] {$2$};
\draw (0,6) -- (1,6) node [midway, above] {$1$};
\draw[fill=black] (0,5) circle (3pt) (1,5) circle (3pt);
\draw (0,6) node {$\times$} (1,6) node {$\times$};
\draw (2,5) rectangle (7,6) node [midway] {$C_2$};
\draw (2,5) -- (3,5) node [midway, below] {$3$};
\draw (3,5) -- (4,5) node [midway, below] {$5$};
\draw (4,5) -- (5,5) node [midway, below] {$1$};
\draw (5,5) -- (6,5) node [midway, below] {$6$};
\draw (6,5) -- (7,5) node [midway, below] {$7$};
\draw (2,6) circle (3pt) (3,5) circle (3pt);
\draw[fill=black] (6,5) circle (3pt) (7,5) circle (3pt);
\draw (4,5) node {$\times$} (5,5) node {$\times$} (6,6) node {$\times$};
\draw[fill=black] (7,6) circle (3pt) (2,6) circle (3pt) (3,6) circle (3pt);
\draw (4,6) circle (3pt) (5,6) circle (3pt);
\draw (2,6) -- (3,6) node [midway, above] {$2$};
\draw (3,6) -- (4,6) node [midway, above] {$3$};
\draw (4,6) -- (5,6) node [midway, above] {$4$};
\draw (5,6) -- (6,6) node [midway, above] {$5$};
\draw (6,6) -- (7,6) node [midway, above] {$6$};
\draw (8,5) rectangle (9,6) node [midway] {$C_3$};
\draw (8,5) -- (9,5) node [midway, below] {$4$};
\draw (8,6) -- (9,6) node [midway, above] {$7$};
\draw[fill=black] (8,6) circle (3pt) (9,6) circle (3pt);
\draw (8,5) circle (3pt) (9,5) circle (3pt);

\draw (0,2.5) rectangle (4,3.5) node [midway] {$C_1$};
\draw (0,2.5) -- (1,2.5) node [midway, below] {$5$};
\draw (1,2.5) -- (2,2.5) node [midway, below] {$4$};
\draw (2,2.5) -- (3,2.5) node [midway, below] {$7$};
\draw (3,2.5) -- (4,2.5) node [midway, below] {$6$};
\draw (0,3.5) -- (4,3.5) node [midway, above] {$1$};
\draw (0,2.5) circle (3pt) (2,2.5) circle (3pt) (4,2.5) circle (3pt);
\draw[fill=black] (0,3.5) circle (3pt) (4,3.5) circle (3pt);
\draw (1,2.5) node {$\times$} (3,2.5) node {$\times$};
\draw (6,2.5) rectangle (10,3.5) node [midway] {$C_3$};
\draw (6,3.5) -- (7,3.5) node [midway, above] {$4$};
\draw (7,3.5) -- (8,3.5) node [midway, above] {$5$};
\draw (8,3.5) -- (9,3.5) node [midway, above] {$6$};
\draw (9,3.5) -- (10,3.5) node [midway, above] {$7$};
\draw (6,2.5) -- (10,2.5) node [midway, below] {$2$};
\draw (7,3.5) circle (3pt) (9,3.5) circle (3pt);
\draw[fill=black] (6,2.5) circle (3pt) (10,2.5) circle (3pt);
\draw (6,3.5) node {$\times$} (8,3.5) node {$\times$} (10,3.5) node {$\times$};

\draw (2,0.5) rectangle (7,1.5) node [midway] {$C_2$};
\draw (2,0.5) -- (3,0.5) node [midway, below] {$3$};
\draw (3,0.5) -- (7,0.5) node [midway, below] {$1$};
\draw (2,1.5) -- (6,1.5) node [midway, above] {$2$};
\draw (6,1.5) -- (7,1.5) node [midway, above] {$3$};
\draw[fill=black] (2,0.5) circle (3pt) (3,0.5) circle (3pt) 
                  (7,0.5) circle (3pt) (2,1.5) circle (3pt)
                  (6,1.5) circle (3pt) (7,1.5) circle (3pt);

\end{tikzpicture}
\end{center}
\caption{List of possible separatrix diagrams in a suitable direction, case
$\Prym[3](2,1,1)$.} 
\label{fig:SuitableDir211} 
\end{figure}
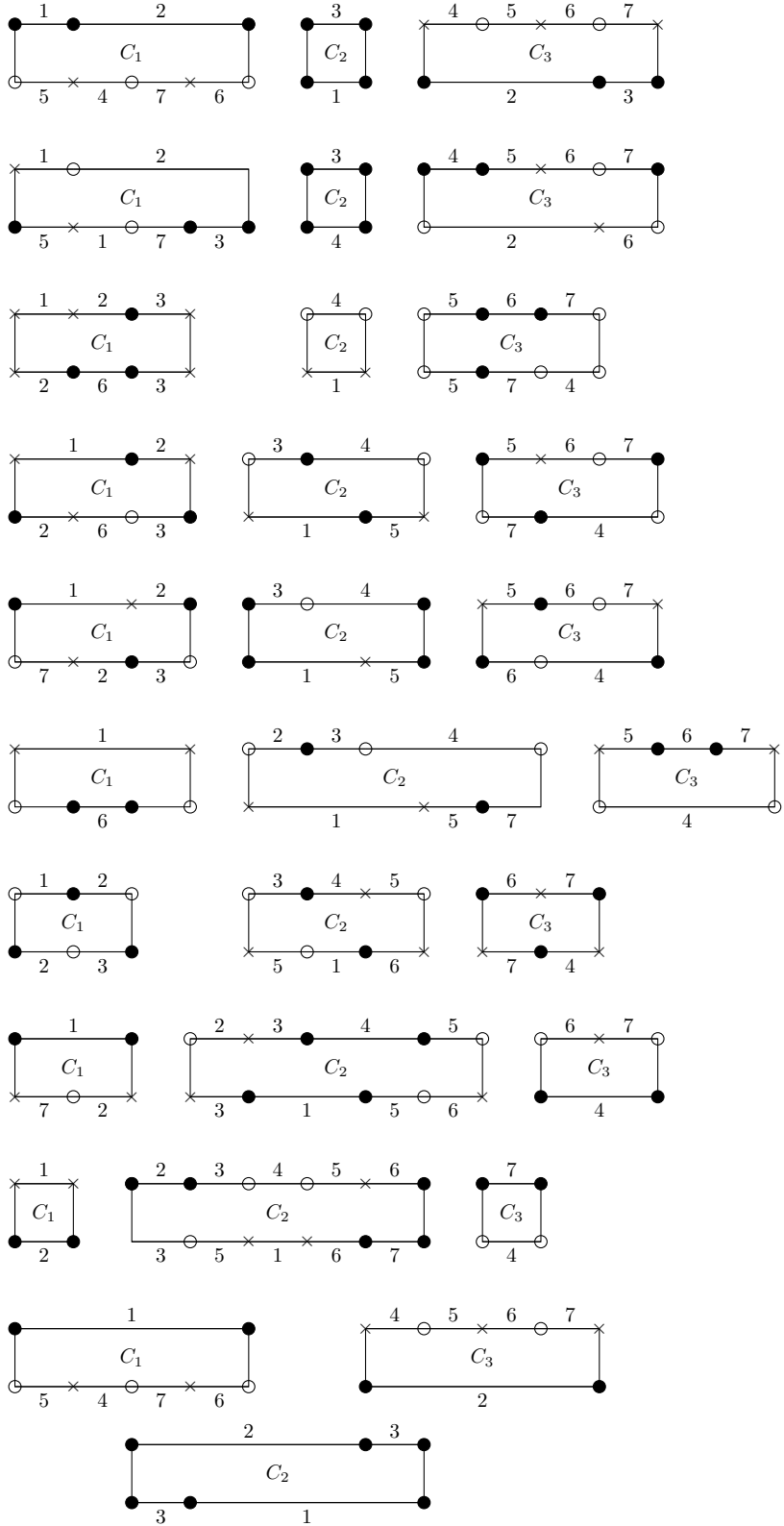
\section{Implementing the algorithm}

\subsection{Results in the case $\Prym[3](2,1,1)$} 
Implementing the algorithm of Section~\ref{sec:implementation} shows that 
there is no solutions for $D_0\in\{2,3,6\}$.
Moreover, for $D_0=33$ there are $3$ possible matrices 
for $M^{\rm red}$. Two matrices correspond to the case where $\beta$ 
crosses the cylinder $C_2$. These can be excluded by Lemma~\ref{le:211crossC1}.
The remaining case is given by 
\bas
&w(C_1)= 1,\quad  &&h(C_1) = 1, \quad &&w(Z_1) = \tfrac{9+\sqrt{33}}2, \quad  
&&h(Z_1) = \tfrac{-3 + \sqrt{33}}{36}\\
&w(C_2) = \tfrac{3+\sqrt{33}}{6}, \quad &&h(C_2) = \tfrac{3 + \sqrt{33}}2, 
\quad  &&w(Z_2) = \tfrac{21 + 3\sqrt{33}}2 \quad &&h(Z_2) = \tfrac13
\eas
and the reduced intersection matrix
$$
M^\red = \left( \begin{matrix} 0 & 6 \\ 3 & 3 \end{matrix} \right).
\quad \text{Hence} \quad
M = \left( \begin{matrix} 0 & 3 & 0 \\ 3 & 3 & 3 \\ 0 & 3 & 0 \end{matrix} \right).
$$

\subsection{Results in the case $\Prym[3](2,2)$}  \label{sec:resPrym22}
The preceding algorithm shows that there is no solutions for $D_0=6$.
Moreover, for $D_0\in\{2,3,33\}$ there are $7$ possible matrices 
for $M^{\rm red}$. More precisely the following proposition holds:
\par
\begin{Prop} \label{prop:redint22}
If $(\Theta_1,\Theta_2)$ is a pair of directions on a primitive Veech surface in $\Prym[3](2,2)$
satisfying above convention then the possible reduced intersection matrices are:
\begin{table}[htb]
\begin{tabular}{|l|c|c|l|c|c|}
  \hline
  K & $M^{\rm red}$ & $r^{\mathrm hor}_2$ & K & $M^{\rm red}$ & $r^{\mathrm hor}_2$ \\
  \hline
  $\QQ[\sqrt{2}]$ & $\left( \begin{matrix} 72 & 48 \\ 24 & 18 \end{matrix} \right)$ & $(\sqrt{2})/2$  & $\QQ[\sqrt{33}]$
& $\left( \begin{matrix} 6 & 24 \\ 12 & 54 \end{matrix} \right)$ & $(3+\sqrt{33})/2$ \\ \hline
  $\QQ[\sqrt{3}]$ & $\left( \begin{matrix} 72 & 24 \\ 12 & 6 \end{matrix} \right)$ & $(-1+\sqrt{3})/2$ &  $\QQ[\sqrt{33}]$
& $\left( \begin{matrix} 6 & 24 \\ 3 & 18 \end{matrix} \right)$ & $(-3+\sqrt{33})/2$   \\ \hline
$\QQ[\sqrt{3}]$  & $\left( \begin{matrix} 72 & 24 \\ 48 & 18 \end{matrix} \right)$ & $(1+\sqrt{3})/2$ &  $\QQ[\sqrt{33}]$
    & $\left( \begin{matrix} 3 & 6 \\ 3 & 0 \end{matrix} \right)$ & $(-3+\sqrt{33})/2$  \\ \hline
$\QQ[\sqrt{3}]$  & $\left( \begin{matrix} 36 & 12 \\ 30 & 12 \end{matrix} \right)$ & $\sqrt{3}$  &&&\\ \hline
\end{tabular}
\end{table}
\end{Prop}
We now discuss case by case each possible intersection matrix.
\par
\subsection{Solutions for the intersection matrices when $D_0=2$}
We consider the  reduced intersection matrix $M_1^\red = \left( \begin{smallmatrix} 72 & 48 \\ 24 & 18 \end{smallmatrix} \right)$.
\bas
&w(C_1)= 1,\quad  &&h(C_1) = 1, \quad &&w(Z_1) = \scriptstyle{72+48\sqrt{2}} , \quad  
&&h(Z_1) =\tfrac{3 - 2\sqrt{2}}{24}  \\
&w(C_2) =\tfrac{\sqrt{2}}{2}, \quad &&h(C_2) = \scriptstyle{2\sqrt{2}}, 
\quad  &&w(Z_2) = \scriptstyle{48 + 36\sqrt{2}} \quad &&h(Z_2) = \tfrac{-4+3\sqrt{2}}{12}
\eas

\subsection{Solutions for the intersection matrices when $D_0=3$}
For the reduced intersection matrix $M_2^\red = \left( \begin{smallmatrix} 72 & 24 \\ 12 & 6 \end{smallmatrix} \right)$ one has
\bas
&w(C_1)= 1,\quad  &&h(C_1) = 1, \quad &&w(Z_1) = \scriptstyle{48+24\sqrt{3}}, \quad  
&&h(Z_1) =\tfrac{2 - \sqrt{3}}{24}  \\
&w(C_2) =\tfrac{-1+\sqrt{3}}{2}, \quad &&h(C_2) =  \scriptstyle{-2+2\sqrt{3}}, 
\quad  &&w(Z_2) = \scriptstyle{12 + 12\sqrt{3}} \quad &&h(Z_2) = \tfrac{-5+3\sqrt{3}}{12}\,.
\eas
For the reduced intersection matrix $M_3^\red = \left( \begin{smallmatrix} 72 & 24 \\ 48 & 18 \end{smallmatrix} \right)$ one has
\bas
&w(C_1)= 1,\quad  &&h(C_1) = 1, \quad &&w(Z_1) = \scriptstyle{168+96\sqrt{3}}, \quad  
&&h(Z_1) =\tfrac{2 - \sqrt{3}}{24}  \\
&w(C_2) = \tfrac{1+\sqrt{3}}{2} , \quad &&h(C_2) =  \scriptstyle{2+2\sqrt{3}} , 
\quad  &&w(Z_2) = \scriptstyle{60 + 36\sqrt{3}} \quad &&h(Z_2) = \tfrac{-5+3\sqrt{3}}{12}\,. 
\eas
For the reduced intersection matrix $M_4^\red = \left( \begin{smallmatrix} 36 & 12 \\ 30 & 12 \end{smallmatrix} \right)$ one has
\bas
&w(C_1)= 1,\quad  &&h(C_1) = 1, \quad &&w(Z_1) =  \scriptstyle{36+20\sqrt{3}}, \quad  
&&h(Z_1) =\tfrac{2 - \sqrt{3}}{12}  \\
&w(C_2) = \scriptstyle{\sqrt{3}}, \quad &&h(C_2) = \scriptstyle{2(\sqrt{3})/3}, 
\quad  &&w(Z_2) = \scriptstyle{12 + 8\sqrt{3}}, \quad &&h(Z_2) = \tfrac{-5+3\sqrt{3}}{6}\,. 
\eas
\subsection{Solutions for the intersection matrices when $D_0=33$}
For the reduced intersection matrix $M_5^\red = \left( \begin{smallmatrix} 6 & 24 \\ 12 & 54 \end{smallmatrix} \right)$ one has
\bas
&w(C_1)= 1,\quad  &&h(C_1) = 1, \quad &&w(Z_1) =  \scriptstyle{12+2\sqrt{33}}, \quad  
&&h(Z_1) = \tfrac{6 - \sqrt{33}}{6}, \\
&w(C_2) = \tfrac{3+\sqrt{33}}{2}, \quad &&h(C_2) = \tfrac{3+\sqrt{33}}{6}\, 
\quad  &&w(Z_2) = \scriptstyle{51 + 9\sqrt{33}}, \quad &&h(Z_2) = \tfrac{-5+\sqrt{33}}{12}\,. 
\eas
For the reduced intersection matrix $M_6^\red = \left( \begin{smallmatrix} 6 & 24 \\ 3 & 18 \end{smallmatrix} \right)$ one has
\bas
&w(C_1)= 1,\quad  &&h(C_1) = 1, \quad &&w(Z_1) = \tfrac{9+\sqrt{33}}{2}   \quad  
&&h(Z_1) = \tfrac{6 - \sqrt{33}}{6}\  \\
&w(C_2) = \tfrac{-3+\sqrt{33}}{2},\quad &&h(C_2) = \tfrac{-3+\sqrt{33}}{6}, 
\quad  &&w(Z_2) = \scriptstyle{15 + 3\sqrt{33}}, \quad &&h(Z_2) = \tfrac{-5+\sqrt{33}}{12}\,. 
\eas
For the reduced intersection matrix $M_7^\red = \left( \begin{smallmatrix} 3 & 6 \\ 3 & 0 \end{smallmatrix} \right)$ one has
\bas
&w(C_1)= 1,\quad  &&h(C_1) = 1, \quad &&w(Z_1) =  \tfrac{3+\sqrt{33}}{2} , \quad  
&&h(Z_1) = \tfrac{-3 + \sqrt{33}}{12}\\
&w(C_2) = \tfrac{-3+\sqrt{33}}{2}, \quad &&h(C_2) = \tfrac{-3+\sqrt{33}}{6}\, 
\quad  &&w(Z_2) = 6, \quad &&h(Z_2) = \tfrac{7-\sqrt{33}}{12}\,. 
\eas

\section{Non-existence in $\Prym[3](2,1,1)$}

We can now complete the proof of non-existence, using again the list
of configurations in this stratum.
\par
\begin{proof}[Proof of Theorem~\ref{thm:main:2}]
>From the intersection matrix we deduce that $C_1$ and $C_3$ 
consists of intersection points with $Z_2$ only. Since $9$ is
odd, the cylinder $C_2$ has a fixed point of $\rho$ in the
center of one of the rectangles, necessarily an intersection
with $Z_2$. In Figure~\ref{fig:Remaining211} this is the leftmost
rectangle of the middle strip.  In $C_2$ the other two intersection 
rectangles with
$Z_2$ are symmetric with respect to this rectangle. Suppose with loss of 
generality that there is a double zero on the bottom $C_1$.
\par
We now use that no possibility for the relative period between
a simple zero and a double zero (compare the list in the previous section)
is rational. Consequently, the lower boundary of $C_1$ does not contain
a simple zero.  By inspection of Figure~\ref{fig:SuitableDir211} we conclude
that the lower boundary of $C_1$ has a single saddle connection. This implies
that the three occurrences of $Z_2$-rectangles in $C_2$ are adjacent, as drawn in
Figure~\ref{fig:Remaining211}. 
\begin{figure}[htbp]
\begin{center}
\begin{tikzpicture}
\draw (0,0) rectangle (1,1) node [midway] {$Z_2$};
\draw (1,0) rectangle (2,1) node [midway] {$Z_2$};
\draw (2,0) rectangle (3,1) node [midway] {$Z_2$};
\draw (1,1) rectangle (2,2) node [midway] {$Z_2$};
\draw (2,1) rectangle (3,2) node [midway] {$Z_2$};
\draw (0,0) -- (1,0) node [midway, below] {$4$};
\draw (1,0) -- (2,0) node [midway, below] {$5$};
\draw (2,0) -- (3,0) node [midway, below] {$6$};
\draw (0,1) -- (1,1) node [midway, above] {$B$};
\draw (3,1) rectangle (4,2);
\draw (4,1) rectangle (5,2);
\draw (5,1) rectangle (6,2);
\draw (6,1) rectangle (7,2);
\draw (7,1) rectangle (8,2);
\draw (8,1) rectangle (9,2);
\draw (9,1) rectangle (10,2) node [midway] {$Z_2$};
\draw (9,2) -- (10,2) node [midway, above] {$A'$};
\draw (9,1) -- (10,1) node [midway, below] {$B'$};
\draw[fill=black] (3,1) circle (3pt);
\draw[fill=black] (3,2) circle (3pt);
\draw (0,2) rectangle (1,3) node [midway] {$Z_2$};
\draw (1,2) rectangle (2,3) node [midway] {$Z_2$};
\draw (2,2) rectangle (3,3) node [midway] {$Z_2$};
\draw (0,3) -- (1,3) node [midway, above] {$1$};
\draw (1,3) -- (2,3) node [midway, above] {$2$};
\draw (2,3) -- (3,3) node [midway, above] {$3$};
\draw (0,2) -- (1,2) node [midway, below] {$A$};
\end{tikzpicture}
\end{center}
\caption{Ruling out the remaining case in $\Prym[3](2,1,1)$.} 
\label{fig:Remaining211}
\end{figure}
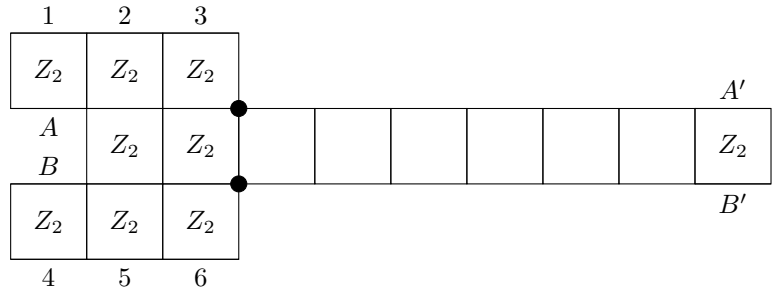
To ensure that the total angle at the double zero (indicated by a black circle)
does not exceed $6\pi$ the two rectangles with the symbol $1$ resp.\ $2$ have to
be glued. This contradicts that the unlabeled rest of $C_2$ consists of just
two vertical cylinders $Z_1$ and $Z_3$.
\end{proof}

\section{Effective finiteness in the 
locus  $\Prym[3](2,2)$} \label{sec:model22}

In this section we complete the proof of Theorem~\ref{thm:main}
by giving a practically feasible algorithm to compile a short finite
list of remaining candidate surfaces. The algorithm proceeds case
by case according to the list of possible separatrix diagrams in 
Figure~\ref{fig:SuitableDir22}. We give all the details for a specific
case, the fourth diagram 'SD4' in this figure, see also the horizontal
direction of the surface in Figure~\ref{fig:model:4} below. We summarize 
the output of the algorithm in the remaining cases.
\par
\begin{figure}[htbp]
\begin{tikzpicture}[scale=0.6]
\fill[fill=yellow!80!black!20,even  odd rule]  (-1,0)  rectangle (2,3);
\draw (0,0) --(-1,0) -- (-1,3) -- (0,3)  -- (0.5,4) --  (5.5,4) --  (5,3) -- (2,3) --
(2,0) -- (5,0) -- (4.5,-1) -- (-0.5,-1) -- cycle; 
\draw (0,0) -- (2,0) (0,3) -- (2,3);

 \draw[thick,  dashed, ->,  >=angle 45]
(-1,0)  .. controls (-0.5,1.5) ..   (-1,3);  \draw[thick, dashed, ->, >= angle 45]
(3.5,3) -- (4,4); \draw[thick,  dashed, ->, >= angle 45](-0.25,-0.5)
-- (2,-0.5);    \draw[thick,   dashed]   (2,-0.5)    --   (4.75,-0.5);
\draw[thick, dashed, ->, >= angle 45] (3,-1) -- (3.5,0);
 
 \draw[thick,dashed,  ->,  >=  angle  45] (0.25,3.5)  --  (3,3.5);  \draw[thick,
  dashed] (1,3.5) -- (5.25,3.5);

\draw[->,>= angle 45, thick,  dashed] (-1,1.5) -- (1,1.5); \draw[thick,
  dashed]  (1,1.5) --  (2,1.5);
    
\filldraw[fill=white,  draw=black]  (-1,0)  circle (2pt)  (2,0)  circle
(2pt) (2.5,4) circle (2pt) (1.5,-1) circle (2pt)
(-1,3) circle  (2pt) (2,3) circle  (2pt) ;

\filldraw[fill=black,  draw=black]  (0,0) circle  (2pt) (0,3) circle  (2pt)
(0.5,4) circle (2pt) (5.5,4) circle (2pt) (5,3) circle  (2pt) (-0.5,-1) circle (2pt) (4.5,-1) circle
(2pt) (5,0) circle  (2pt);

\draw  (0,1.5) node[above]  {$\scriptstyle \alpha_1$}  (-1,1.5) node[left]
      {$\scriptstyle \beta_1$}  (2.75, 2.5) node[right] {$\scriptstyle
        \beta_{2,1}$} (6,3) node[above] {$\scriptstyle \alpha_{2,1}$}
      (-0.25,-0.5)  node[left]  {$\scriptstyle \alpha_{2,2}$}  (3.5,0)
      node[above] {$\scriptstyle \beta_{2,2}$} (1,3)
      node[below] {$\scriptstyle \gamma=s\cdot \lambda$};
\end{tikzpicture}
\caption{ 
\label{fig:model:4}
If       $\alpha_2:=\alpha_{2,1}+\alpha_{2,2}$ and $\beta_2:=\beta_{2,1}+\beta_{2,2}$
then  $\{\alpha_1,\beta_1,\alpha_2,\beta_2\}$     is     a  symplectic  basis     of
  $H_{1}(X,\ZZ)^{-}$.
}
\end{figure}
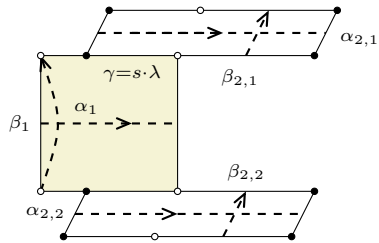

\subsection{Finding arithmetic surfaces with a given reduced intersection matrix}

We have already seen that the finiteness of the possibilities for the
parameters $w(C_i),h(C_i),w(Z_i),h(Z_i)$ as well as the reduced intersection matrix 
gives at most a finite number of possibilities for the flat surface $(X,\omega)$.
We also know that~$(X,\omega)$ is tiled by rectangles $R_{i,j}$ consisting of t
he intersection of the cylinders $C_i$ and $Z_j$ for $i,j \in \{1,2\}$.
\par
To actually give a good practical upper bound for the number of possibilities
we proceed as follows. First, we use a prototype to present the surface
with the given separatrix diagrams in the horizontal direction in a standard
form. Second, we loop over possible 'arithmetic surfaces' that encode the 
adjacency of the rectangles $R_{i,j}$.
\par
The first step is given by the following proposition, whose proof is
completely parallel to that of~\cite[Proposition 4.2]{LNWeierstrass}.
\par
\begin{Prop} \label{prop:proto}
Let $(X,\omega) \in \Omega E_D(2,2)$  be a Prym eigenform with horizontal
separatrix diagram as in 'SD4', equipped with
loops $\alpha_{i,j}$ and $\beta_{i,j}$ as   presented in Figure~\ref{fig:model:4}.
Then after applying a suitable element in the upper triangular group
there exist $(w,h,t,e) \in \ZZ^{4}$ and $s\in (0,1)$ such that \medskip
%
\begin{itemize}
\item         the          tuple         $(w,h,t,e)$         satisfies
  $(\mathcal{P}_D) := \left\{\begin{array}{l}        w>0,h>0,\;        0\leq
  t<\gcd(w,h),\\    \gcd   (w,h,t,e)    =1,\\    D=e^2+8w   h,\\    0<
  \lambda:=\frac{e+\sqrt{D}}{2}< w  \\
\end{array}
\right.$,
\item there exists  a generator $T$ of $\mathcal O_D$ with the property 
$T^*(\omega)=\lambda\omega$ that is represented
in  the  basis  $\{\alpha_1,\beta_1,\alpha_2,\beta_2\}$ by 
$\left( 
\begin{smallmatrix}
  e & 0 & 2w & 2t \\ 0 & e & 0 & 2h \\ h & -t & 0 & 0 \\ 0 & w & 0 & 0
  \\
\end{smallmatrix}%
\right)$.
\item \label{normalize:A} and such that with this choice  coordinates
$$\left\{ \begin{array}{l}
  \omega(\ZZ\alpha_1+\ZZ\beta_1)=\lambda\cdot \ZZ^2,
  \\ \omega(\ZZ\alpha_{2,1}+\ZZ\beta_{2,2})=\omega(\ZZ\alpha_{2,2}+\ZZ\beta_{2,2})=\ZZ(w,0)+\ZZ(t,h),\\
  \omega(\gamma)=(s\lambda,0).
\end{array}
\right. 
$$
%
\end{itemize}
\noindent Conversely,  let $(X,\omega) \in \Prym[3](2,2)$ having the above decomposition such that there exists 
$(w,h,t,e)  \in \ZZ^4$ verifying $(\mathcal{P}_D)$  such  that,  after  normalizing by  $\GL^+(2,\RR)$,  the
conditions are satisfied, then $(X,\omega)\in \Omega E_D(2,2)$.
\par
\end{Prop}
\par
We refer to the parameters~$s$ and~$t$ as the {\em slit} and {\em twist} respectively.
\par
For the second step we construct all square-tiled surfaces in $\Prym[3](2,2)$
with the horizontal separatrix diagram as in 'SD4' and such that the associated 
intersection matrix is a given matrix $M^{\rm red}$ as listed in 
Section~\ref{sec:resPrym22}. More concretely, let $a_i$ be the number of squares
of the horizontal cylinder~$C_i$ and $b_j$ be the number of squares of 
the vertical cylinder~$Z_j$. Then obviously
$$
\begin{array}{ll}
a_1 = M^{\rm red}_{11} + \frac1{2}M^{\rm red}_{12} & a_2 = 2\cdot M^{\rm red}_{21}+M^{\rm red}_{22} \\
b_1 = M^{\rm red}_{11} + M^{\rm red}_{21} & b_2 = M^{\rm red}_{12}+M^{\rm red}_{22}
\end{array}
$$
Let $\ell_i \in \NN$ for $i=0,\dots,6$ be the length of the horizontal saddle 
connections $\gamma_i$ with label~$i$ in Figure~\ref{fig:SuitableDir22} (still 
diagram 'SD4')  in the corresponding square-tiled surface. We determine all 
possible square-tiled surfaces by running a loop over all 
$\ell_3 \in [0..\min(a_1,a_2)]$, all
twist parameters $T_1 \in [0..a_1]$, and all $T_2\in [0..a_2]$ (in the cylinder~$C_i$, 
with respect to the leftmost singularities in the figure) and then checking 
whether the vertical direction on the surface produces the given reduced 
intersection matrix. We refer to the square-tiled surfaces constructed
in this way as the {\em arithmetic surfaces} underlying $(X,\omega) \in 
\Omega E_D(2,2)$ with suitable horizontal and vertical direction.
\par
In order to convert such an arithmetic surface into a candidate for
a Veech surface in $\Omega E_D(2,2)$ in the normalization given in
Proposition~\ref{prop:proto}, we replace each square by a rectangle
$R_{i,j}$ according to the horizontal and vertical cylinder the square lies in.
Next, we convert the twists $T_i$ on the arithmetic surface into twists
on the candidate for a Veech surface. In fact, if $T_k$ shifts by
$N^k_{i,j}$ squares that correspond to rectangles of type $R_{i,j}$, 
then obviously 
$$
t_k \= h(Z_1)\cdot \sum_{j=1}^3 (N^k_{j,1}+N^k_{j,3}) + h(Z_2)\cdot \sum_{j=1}^3 N^k_{j,2}\,.
$$
Finally, we scale by an upper triangular matrix so that the central horizontal
cylinder becomes a square. 
\par
Once the surface is constructed, we check that the vertical direction 
is admissible {\em i.e.} there is an 
invariant saddle connection that represents a relative period and that 
is contained in $C_2$. This rules out in practice a large number of surfaces.
\par
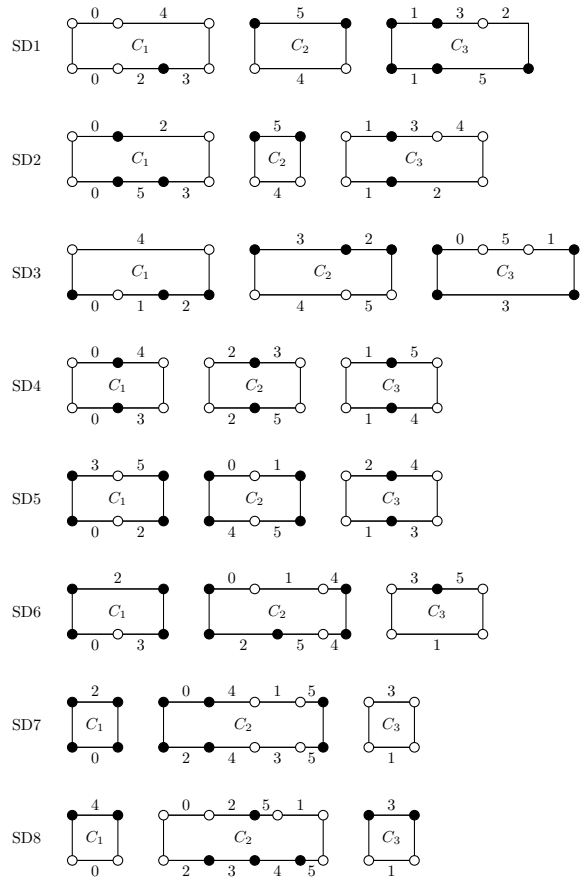
\begin{figure}[htbp]
\begin{center}
\begin{tikzpicture} [scale=0.6, every node/.style={transform shape}]
%
\node (M1) at (0,18.5) {SD1};  
\draw (1,18) rectangle (4,19)  node [midway] {$C_1$};
\draw (2,19) -- (4,19) node [midway, above] {$4$};
\draw (1,19) -- (2,19) node [midway, above] {$0$};
\draw (3,18) -- (4,18) node [midway, below] {$3$};
\draw (2,18) -- (3,18) node [midway, below] {$2$};
\draw (1,18) -- (2,18) node [midway, below] {$0$};
\draw[fill=white] (1,18) circle (3pt) (2,18) circle (3pt) (4,18) circle (3pt)
(1,19) circle (3pt) (2,19) circle (3pt) (4,19) circle (3pt);
\draw[fill=black] (3,18) circle (3pt);
\draw (5,18) rectangle (7,19)  node [midway] {$C_2$};
\draw (5,19) -- (7,19) node [midway, above] {$5$};
\draw (5,18) -- (7,18) node [midway, below] {$4$};
\draw[fill=white] (5,18) circle (3pt) (7,18) circle (3pt);
\draw[fill=black] (5,19) circle (3pt) (7,19) circle (3pt);
\draw (8,18) rectangle (11,19)  node [midway] {$C_3$};
\draw (10,19) -- (11,19) node [midway, above] {$2$};
\draw (9,19) -- (10,19) node [midway, above] {$3$};
\draw (8,19) -- (9,19) node [midway, above] {$1$};
\draw (9,18) -- (11,18) node [midway, below] {$5$};
\draw (8,18) -- (9,18) node [midway, below] {$1$};

\draw[fill=white] (10,19) circle (3pt);
\draw[fill=black] (8,18) circle (3pt) (9,18) circle (3pt) (11,18) circle (3pt)
                  (8,19) circle (3pt) (9,19) circle (3pt);
%
\node (M2) at (0,16) {SD2};  
\draw (1,15.5) rectangle (4,16.5)  node [midway] {$C_1$};
\draw (2,16.5) -- (4,16.5) node [midway, above] {$2$};
\draw (1,16.5) -- (2,16.5) node [midway, above] {$0$};
\draw (3,15.5) -- (4,15.5) node [midway, below] {$3$};
\draw (2,15.5) -- (3,15.5) node [midway, below] {$5$};
\draw (1,15.5) -- (2,15.5) node [midway, below] {$0$};
\draw[fill=white] (1,15.5) circle (3pt) (4,15.5) circle (3pt)
(1,16.5) circle (3pt) (4,16.5) circle (3pt);
\draw[fill=black] (2,15.5) circle (3pt) (3,15.5) circle (3pt) (2,16.5) circle (3pt);
\draw (5,15.5) rectangle (6,16.5)  node [midway] {$C_2$};
\draw (5,16.5) -- (6,16.5) node [midway, above] {$5$};
\draw (5,15.5) -- (6,15.5) node [midway, below] {$4$};
\draw[fill=white] (5,15.5) circle (3pt) (6,15.5) circle (3pt);
\draw[fill=black] (5,16.5) circle (3pt) (6,16.5) circle (3pt);
\draw (7,15.5) rectangle (10,16.5)  node [midway] {$C_3$};
\draw (9,16.5) -- (10,16.5) node [midway, above] {$4$};
\draw (8,16.5) -- (9,16.5) node [midway, above] {$3$};
\draw (7,16.5) -- (8,16.5) node [midway, above] {$1$};
\draw (8,15.5) -- (10,15.5) node [midway, below] {$2$};
\draw (7,15.5) -- (8,15.5) node [midway, below] {$1$};
\draw[fill=white] (7,15.5) circle (3pt) (10,15.5) circle (3pt)
                  (7,16.5) circle (3pt) (9,16.5) circle (3pt) (10,16.5) circle (3pt);
\draw[fill=black] (8,15.5) circle (3pt) (8,16.5) circle (3pt);
\node (M3) at (0,13.5) {SD3};  
\draw (1,13) rectangle (4,14)  node [midway] {$C_1$};

\draw (1,14) -- (4,14) node [midway, above] {$4$};
\draw (3,13) -- (4,13) node [midway, below] {$2$};
\draw (2,13) -- (3,13) node [midway, below] {$1$};
\draw (1,13) -- (2,13) node [midway, below] {$0$};
\draw[fill=white] (1,14) circle (3pt) (4,14) circle (3pt)
(2,13) circle (3pt);
\draw[fill=black] (1,13) circle (3pt) (3,13) circle (3pt) (4,13) circle (3pt);
\draw (5,13) rectangle (8,14)  node [midway] {$C_2$};
\draw (7,14) -- (8,14) node [midway, above] {$2$};
\draw (5,14) -- (7,14) node [midway, above] {$3$};
\draw (7,13) -- (8,13) node [midway, below] {$5$};
\draw (5,13) -- (7,13) node [midway, below] {$4$};
\draw[fill=black] (5,14) circle (3pt) (7,14) circle (3pt) (8,14) circle (3pt);
\draw[fill=white] (5,13) circle (3pt) (7,13) circle (3pt) (8,13) circle (3pt);
\draw (9,13) rectangle (12,14)  node [midway] {$C_3$};
\draw (11,14) -- (12,14) node [midway, above] {$1$};
\draw (10,14) -- (11,14) node [midway, above] {$5$};
\draw (9,14) -- (10,14) node [midway, above] {$0$};
\draw (9,13) -- (12,13) node [midway, below] {$3$};
\draw[fill=black] (9,13) circle (3pt) (9,14) circle (3pt) (12,13) circle (3pt) (12,14) circle (3pt);
\draw[fill=white] (10,14) circle (3pt) (11,14) circle (3pt);
%
\node (M4) at (0,11) {SD4};  
\draw (1,10.5) rectangle (3,11.5)  node [midway] {$C_1$};
\draw (2,11.5) -- (3,11.5) node [midway, above] {$4$};
\draw (1,11.5) -- (2,11.5) node [midway, above] {$0$};
\draw (2,10.5) -- (3,10.5) node [midway, below] {$3$};
\draw (1,10.5) -- (2,10.5) node [midway, below] {$0$};
\draw[fill=white] (1,10.5) circle (3pt) (3,10.5) circle (3pt)
(1,11.5) circle (3pt) (3,11.5) circle (3pt);
\draw[fill=black] (2,10.5) circle (3pt) (2,11.5) circle (3pt);
\draw (4,10.5) rectangle (6,11.5)  node [midway] {$C_2$};
\draw (5,11.5) -- (6,11.5) node [midway, above] {$3$};
\draw (4,11.5) -- (5,11.5) node [midway, above] {$2$};
\draw (5,10.5) -- (6,10.5) node [midway, below] {$5$};
\draw (4,10.5) -- (5,10.5) node [midway, below] {$2$};
\draw[fill=white] (4,10.5) circle (3pt) (6,10.5) circle (3pt)
(4,11.5) circle (3pt) (6,11.5) circle (3pt);
\draw[fill=black] (5,10.5) circle (3pt) (5,11.5) circle (3pt);
\draw (7,10.5) rectangle (9,11.5)  node [midway] {$C_3$};
\draw (8,11.5) -- (9,11.5) node [midway, above] {$5$};
\draw (7,11.5) -- (8,11.5) node [midway, above] {$1$};
\draw (8,10.5) -- (9,10.5) node [midway, below] {$4$};
\draw (7,10.5) -- (8,10.5) node [midway, below] {$1$};
\draw[fill=white] (7,10.5) circle (3pt) (9,10.5) circle (3pt)
(7,11.5) circle (3pt) (9,11.5) circle (3pt);
\draw[fill=black] (8,10.5) circle (3pt) (8,11.5) circle (3pt);
%
\node (M5) at (0,8.5) {SD5};  
\draw (1,8) rectangle (3,9)  node [midway] {$C_1$};
\draw (2,9) -- (3,9) node [midway, above] {$5$};
\draw (1,9) -- (2,9) node [midway, above] {$3$};
\draw (2,8) -- (3,8) node [midway, below] {$2$};
\draw (1,8) -- (2,8) node [midway, below] {$0$};
\draw[fill=black] (1,8) circle (3pt) (3,8) circle (3pt)
(1,9) circle (3pt) (3,9) circle (3pt);
\draw[fill=white] (2,8) circle (3pt) (2,9) circle (3pt);
\draw (4,8) rectangle (6,9)  node [midway] {$C_2$};
\draw (5,9) -- (6,9) node [midway, above] {$1$};
\draw (4,9) -- (5,9) node [midway, above] {$0$};
\draw (5,8) -- (6,8) node [midway, below] {$5$};
\draw (4,8) -- (5,8) node [midway, below] {$4$};
\draw[fill=black] (4,8) circle (3pt) (6,8) circle (3pt)
(4,9) circle (3pt) (6,9) circle (3pt);
\draw[fill=white] (5,8) circle (3pt) (5,9) circle (3pt);
\draw (7,8) rectangle (9,9)  node [midway] {$C_3$};
\draw (8,9) -- (9,9) node [midway, above] {$4$};
\draw (7,9) -- (8,9) node [midway, above] {$2$};
\draw (8,8) -- (9,8) node [midway, below] {$3$};
\draw (7,8) -- (8,8) node [midway, below] {$1$};
\draw[fill=black] (8,8) circle (3pt) (8,9) circle (3pt);
\draw[fill=white] (7,8) circle (3pt) (9,8) circle (3pt)
                  (7,9) circle (3pt) (9,9) circle (3pt);
%
\node (M6) at (0,6) {SD6};  
\draw (1,5.5) rectangle (3,6.5)  node [midway] {$C_1$};
\draw (1,6.5) -- (3,6.5) node [midway, above] {$2$};
\draw (2,5.5) -- (3,5.5) node [midway, below] {$3$};
\draw (1,5.5) -- (2,5.5) node [midway, below] {$0$};
\draw[fill=black] (1,5.5) circle (3pt) (3,5.5) circle (3pt)
                  (1,6.5) circle (3pt) (3,6.5) circle (3pt);
\draw[fill=white] (2,5.5) circle (3pt);
\draw (4,5.5) rectangle (7,6.5)  node [midway] {$C_2$};
\draw (6.5,6.5) -- (7,6.5) node [midway, above] {$4$};
\draw (5,6.5) -- (6.5,6.5) node [midway, above] {$1$};
\draw (4,6.5) -- (5,6.5) node [midway, above] {$0$};
\draw (6.5,5.5) -- (7,5.5) node [midway, below] {$4$};
\draw (5.5,5.5) -- (6.5,5.5) node [midway, below] {$5$};
\draw (4,5.5) -- (5.5,5.5) node [midway, below] {$2$};
\draw[fill=black] (4,5.5) circle (3pt) (5.5,5.5) circle (3pt) (7,5.5) circle (3pt)
                  (4,6.5) circle (3pt) (7,6.5) circle (3pt);
\draw[fill=white] (5,6.5) circle (3pt) (6.5,6.5) circle (3pt) (6.5,5.5) circle (3pt);
\draw (8,5.5) rectangle (10,6.5)  node [midway] {$C_3$};
\draw (9,6.5) -- (10,6.5) node [midway, above] {$5$};
\draw (8,6.5) -- (9,6.5) node [midway, above] {$3$};
\draw (8,5.5) -- (10,5.5) node [midway, below] {$1$};
\draw[fill=black] (9,6.5) circle (3pt);
\draw[fill=white] (8,6.5) circle (3pt) (10,6.5) circle (3pt) 
                  (8,5.5) circle (3pt) (10,5.5) circle (3pt);
%
\node (M7) at (0,3.5) {SD7};
\draw (1,3) rectangle (2,4)  node [midway] {$C_1$};
\draw (1,4) -- (2,4) node [midway, above] {$2$};
\draw (1,3) -- (2,3) node [midway, below] {$0$};
\draw[fill=black] (1,4) circle (3pt) (2,4) circle (3pt)
                  (1,3) circle (3pt) (2,3) circle (3pt);
\draw (3,3) rectangle (6.5,4) node [midway] {$C_2$};
\draw (6,4) -- (6.5,4) node [midway, above] {$5$};
\draw (5,4) -- (6,4) node [midway, above] {$1$};
\draw (4,4) -- (5,4) node [midway, above] {$4$};
\draw (3,4) -- (4,4) node [midway, above] {$0$};
\draw (3,3) -- (4,3) node [midway, below] {$2$};
\draw (4,3) -- (5,3) node [midway, below] {$4$};
\draw (5,3) -- (6,3) node [midway, below] {$3$};
\draw (6,3) -- (6.5,3) node [midway, below] {$5$};
\draw[fill=black] (3,3) circle (3pt) (4,3) circle (3pt) (6.5,3) circle (3pt)
                  (3,4) circle (3pt) (4,4) circle (3pt) (6.5,4) circle (3pt);
\draw[fill=white] (5,3) circle (3pt) (6,3) circle (3pt)
                  (5,4) circle (3pt) (6,4) circle (3pt);
\draw (7.5,3) rectangle (8.5,4)  node [midway] {$C_3$};
\draw (7.5,4) -- (8.5,4) node [midway, above] {$3$};
\draw (7.5,3) -- (8.5,3) node [midway, below] {$1$};
\draw[fill=white] (7.5,4) circle (3pt) (8.5,4) circle (3pt)
                  (7.5,3) circle (3pt) (8.5,3) circle (3pt);                  
%
\node (M8) at (0,1) {SD8}; 
\draw (1,0.5) rectangle (2,1.5)  node [midway] {$C_1$};
\draw (1,1.5) -- (2,1.5) node [midway, above] {$4$};
\draw (1,.5) -- (2,.5) node [midway, below] {$0$};
\draw[fill=black] (1,1.5) circle (3pt) (2,1.5) circle (3pt);
\draw[fill=white] (1,0.5) circle (3pt) (2,0.5) circle (3pt);
\draw (3,0.5) rectangle (6.5,1.5) node [midway] {$C_2$};
\draw (5.5,1.5) -- (6.5,1.5) node [midway, above] {$1$};
\draw (5,1.5) -- (5.5,1.5) node [midway, above] {$5$};
\draw (4,1.5) -- (5,1.5) node [midway, above] {$2$};
\draw (3,1.5) -- (4,1.5) node [midway, above] {$0$};
\draw (3,.5) -- (4,.5) node [midway, below] {$2$};
\draw (4,.5) -- (5,.5) node [midway, below] {$3$};
\draw (5,.5) -- (6,.5) node [midway, below] {$4$};
\draw (6,.5) -- (6.5,.5) node [midway, below] {$5$};
\draw[fill=black] (4,.5) circle (3pt) (5,.5) circle (3pt)
                  (6,.5) circle (3pt) (5,1.5) circle (3pt);
\draw[fill=white] (3,0.5) circle (3pt) (6.5,0.5) circle (3pt)
                  (3,1.5) circle (3pt) (4,1.5) circle (3pt) 
                  (5.5,1.5) circle (3pt) (6.5,1.5) circle (3pt);
\draw (7.5,0.5) rectangle (8.5,1.5)  node [midway] {$C_3$};
\draw (7.5,1.5) -- (8.5,1.5) node [midway, above] {$3$};
\draw (7.5,.5) -- (8.5,.5) node [midway, below] {$1$};
\draw[fill=black] (7.5,1.5) circle (3pt) (8.5,1.5) circle (3pt);
\draw[fill=white] (7.5,0.5) circle (3pt) (8.5,0.5) circle (3pt);

\end{tikzpicture}
\end{center}
\caption{List of possible separatrix diagrams in a suitable direction, 
case $\Prym[3](2,2)$.} 
\label{fig:SuitableDir22} 
\end{figure}
\par


\subsection{Output of the algorithm}

For instance for the first intersection matrix $M^{\rm red}_1$ and the
trace field $\QQ[\sqrt{2}]$, we found $228$ arithmetic surfaces, hence $228$
candidates surfaces. Only $6$ solutions have a vertical admissible direction.
Some of them give the same prototype, in fact there are three different
prototype, as listed in Table~\ref{cap:SD4}. The other intersection matrices
are treated in the same way.
\par
\begin{table}
$$
\begin{array}{|c|c|c|c|c|}
\hline
\multirow{1}{*}{Reduced matrix} & \textrm{\# Arithm.\ surf.} & \textrm{Prototypes } (w,h,t,e) & \textrm{slits} &\textrm{disc.} \\
\hline \hline 
  \multirow{3}{*}{$\left( \begin{smallmatrix} 72 & 48 \\ 24 & 18 \end{smallmatrix} \right)$}&  \multirow{3}{*}{228} &  (4, 1, 0, 0) & \tfrac{3+2\sqrt{2}}{6} & 32\\
    && (12, 3, 1, 0)  & \tfrac{3+2\sqrt{2}}{6} & 288\\
    && (12, 3, 2, 0)  & \tfrac{3+2\sqrt{2}}{6} & 288\\
\hline \hline
  \multirow{3}{*}{$ \left( \begin{smallmatrix} 72 & 24 \\ 12 & 6 \end{smallmatrix} \right)$}&  \multirow{3}{*}{32} &  (4, 1, 0, -4) & \tfrac{4+\sqrt{3}}{6} & 48\\
    && (12, 3, 1, -12)  & \tfrac{4+\sqrt{3}}{6} & 432\\
    && (12, 3, 2, -12)  & \tfrac{4+\sqrt{3}}{6} & 432\\
    \hline
  \multirow{3}{*}{$ \left( \begin{smallmatrix} 72 & 24 \\ 48 & 18 \end{smallmatrix} \right)$}&  \multirow{3}{*}{336} &  (4, 1, 0, 4) & \tfrac{\sqrt{3}}{6} & 48\\
    && (12, 3, 1, 12)  & \tfrac{\sqrt{3}}{6} & 432\\
    && (12, 3, 2, 12)  & \tfrac{\sqrt{3}}{6} & 432\\
    \hline
      \multirow{6}{*}{$\left( \begin{smallmatrix} 36 & 12 \\ 30 & 12 \end{smallmatrix} \right)$}&  \multirow{6}{*}{180} &  (6, 9, 1, 0) & \tfrac{6-\sqrt{3}}{18} & 432\\
    && (6, 9, 2, 0)  & \tfrac{6-\sqrt{3}}{18} & 432\\
    && (6, 9, 1, 0)  & \tfrac{6+\sqrt{3}}{18} & 432\\
    && (6, 9, 2, 0)  & \tfrac{6+\sqrt{3}}{18} & 432\\
    && (12, 18, 1, 0)  & \tfrac{1+\sqrt{3}}{6} & 1728\\
    && (12, 18, 5, 0)  & \tfrac{1+\sqrt{3}}{6} & 1728\\
\hline \hline
  \multirow{1}{*}{$ \left( \begin{smallmatrix} 6 & 24 \\ 12 & 54 \end{smallmatrix} \right)$}&  \multirow{1}{*}{24} &  \textrm{no solutions} &  & \\
    \hline
  \multirow{1}{*}{$ \left( \begin{smallmatrix} 6 & 24 \\ 3 & 18 \end{smallmatrix} \right)$}&  \multirow{1}{*}{0} &  \textrm{no 
  solutions} &  & \\
    \hline
      \multirow{1}{*}{$\left( \begin{smallmatrix} 3 & 6 \\ 3 & 0 \end{smallmatrix} \right)$}&  \multirow{1}{*}{0} &  \textrm{no 
      solutions} &  & \\
    \hline    
\end{array}
$$
\caption{Number of arithmetic candidate surfaces and prototypes with
separatrix diagram SD4.
}
\label{cap:SD4}
\end{table}
\par
\begin{Rem}
\label{rem:noTC}
{\rm In the examples where the twist parameter is zero, one can additionally
check the rationality constraint of ratios of moduli of vertical 
cylinders (see~\cite[Theorem 6.3]{mcmullentor} for similar computations), 
since the cylinder decomposition in the vertical direction is easily computed.
In the three examples in the preceding table with twist zero the moduli 
are indeed not commensurable, thus ruling out these $3$ cases.
}\end{Rem}
\par
In view of the above remark and the results of the table, one concludes that 
there are at most $12$ Teichm\"uller curves for which there is a translation 
surface having a cylinder decomposition with separatrix diagram SD4  
of Figure~\ref{fig:SuitableDir22}.
\par
\medskip
For each of the $7$ remaining possible cylinder decomposition
in Figure~\ref{fig:SuitableDir22} we apply the algorithm described above.
The results are presented in slightly more condensed form in
Table~\ref{cap:algo} below. A quick inspection of this table (combined 
with Remark~\ref{rem:noTC}) reveals the $92$ remaining cases which 
concludes the proof of Theorem~\ref{thm:main}.
\begin{table}
$$
\begin{array}{|c||c|c|c|c|c|c|c|c|}
\hline
\multirow{1}{*}{Reduced matrix $\backslash$ SD} & 1 & 2 & 3 & 4 & 5
& 6 & 7 & 8 \\
\hline \hline
  \multirow{2}{*}{$M_1^{\rm red}=\left( \begin{smallmatrix} 72 & 48 \\ 24 & 18
\end{smallmatrix} \right)$}& 520 &  176 & \multirow{2}{*}{-} & 228 &
342 & \multirow{2}{*}{-}&\multirow{2}{*}{-}&\multirow{2}{*}{-}\\
    &{\bf 12}& {\bf 1}  &  & {\bf 3} & {\bf 4} & &&\\
\hline \hline
  \multirow{2}{*}{$M_2^{\rm red}= \left( \begin{smallmatrix} 72 & 24 \\ 12 & 6
\end{smallmatrix} \right)$}&  108 & 63 &\multirow{2}{*}{-} &
32&88&\multirow{2}{*}{-}&\multirow{2}{*}{-}&\multirow{2}{*}{-}\\
    &{\bf 8}&{\bf 0}&&{\bf 3}&{\bf 0}&&&\\
    \hline
  \multirow{2}{*}{$M_3^{\rm red=} \left( \begin{smallmatrix} 72 & 24 \\ 48 & 18
\end{smallmatrix} \right)$}&
\multirow{2}{*}{-}&\multirow{2}{*}{-}&23&336&290&186&\multirow{2}{*}{-}&\multirow{2}{*}{-}\\
    &&&{\bf 0}&{\bf 3}&{\bf 8}&{\bf 3}&&\\
    \hline
      \multirow{2}{*}{$M_4^{\rm red}= \left( \begin{smallmatrix} 36 & 12 \\ 30 & 12
\end{smallmatrix} \right)$}&
\multirow{2}{*}{-}&\multirow{2}{*}{-}&0&180&48&214&\multirow{2}{*}{-}&\multirow{2}{*}{-}\\
    &&&{\bf 0}&{\bf 6}&{\bf 0}&{\bf 8}&&\\
\hline \hline
  \multirow{2}{*}{$M_5^{\rm red}= \left( \begin{smallmatrix} 6 & 24 \\ 12 & 54
\end{smallmatrix} \right)$}&
\multirow{2}{*}{-}&\multirow{2}{*}{-}&\multirow{2}{*}{-}&24&\multirow{2}{*}{-}&124&392&210\\
  &&&&{\bf 0}&&{\bf 7}&{\bf 18}&{\bf 20}\\
    \hline
  \multirow{2}{*}{$M_6^{\rm red}= \left( \begin{smallmatrix} 6 & 24 \\ 3 & 18
\end{smallmatrix} \right)$}&
\multirow{2}{*}{-}&\multirow{2}{*}{-}&0&0&0&0&\multirow{2}{*}{-}&\multirow{2}{*}{-}\\
  &&&{\bf 0}&{\bf 0}&{\bf 0}&{\bf 0}&&\\
    \hline
      \multirow{2}{*}{$M_7^{\rm red}=\left( \begin{smallmatrix} 3 & 6 \\ 3 & 0
\end{smallmatrix} \right)$}&  \multirow{2}{*}{-}
&\multirow{2}{*}{-}&\multirow{2}{*}{-}&0&0&\multirow{2}{*}{-}&\multirow{2}{*}{-}&\multirow{2}{*}{-}\\
  &&&&{\bf 0}&{\bf 0}&&&\\
    \hline    \hline
\# \textrm{ Candidates}  &{\bf 20}&{\bf 1}&{\bf 0}&{\bf 15}&{\bf
12}&{\bf 18}&{\bf 18}&{\bf 20}\\
  \hline
\end{array}
$$
\caption{Candidates for Teichm\"uller discs for each model. In bold we have 
indicated the number of candidate surfaces with an admissible vertical direction.}
\label{cap:algo}
\end{table}

%
%
%
%
%
%

\printbibliography

\end{document}